\documentclass[11pt]{article}
\usepackage{amsmath,amsthm,amsfonts,amssymb}
\usepackage{alg,natbib,setspace,multirow,multicol,enumitem,epsfig,datetime}
\usepackage{smpl-math}
\newcommand{\rmc}{\mathrm{c}}
\newcommand{\rmu}{\mathrm{u}}
\begin{document}
\date{February 13, 2012}
%\usdate
%
%\mdyyyydate
%\newtimeformat{campmtime}{%
%  \ifthenelse{\value{HOURXII}=0}{12}{\THEHOURXII}%
%  \timeseparator\twodigit\THEMINUTE
%  \ifthenelse{\value{HOUR}<12}{\amname}{\pmname}} 
%\date{DRAFT:\: \today\ \campmtime}
%
\title{Strongly Polynomial Primal-Dual Algorithms for\\ Concave Cost
  Combinatorial Optimization Problems\footnote{This research is based on the
    second author's Ph.D. thesis at the Massachusetts Institute of
    Technology \cite{n:Stratila:2008:COP}.}}
\author{Thomas L. Magnanti\footnote{School of Engineering and Sloan School
    of Management, Massachusetts Institute of Technology, 77 Massachusetts
    Avenue, Room 32-D784, Cambridge, MA 02139. E-mail: {\tt
      magnanti@mit.edu.}}
\and Dan Stratila\footnote{Rutgers Center for Operations Research and
  Rutgers Business School, Rutgers University, 640 Bartholomew Road, Room
  107, Piscataway, NJ 08854. E-mail: {\tt dstrat@rci.rutgers.edu.}}}
\begin{singlespace}
\maketitle
\end{singlespace}
{\abstract{
\begin{singlespace}
We introduce an algorithm design technique for a class of combinatorial
optimization problems with concave costs. This technique yields a strongly
polynomial primal-dual algorithm for a concave cost problem whenever such an
algorithm exists for the fixed-charge counterpart of the problem. For many
practical concave cost problems, the fixed-charge counterpart is a
well-studied combinatorial optimization problem. Our technique preserves
constant factor approximation ratios, as well as ratios that depend only on
certain problem parameters, and exact algorithms yield exact algorithms.

Using our technique, we obtain a new 1.61-approximation algorithm for the
concave cost facility location problem. For inventory problems, we obtain a
new exact algorithm for the economic lot-sizing problem with general concave
ordering costs, and a 4-approximation algorithm for the joint replenishment
problem with general concave individual ordering costs.
\end{singlespace}
}}
%
%
% __________________________________________________________________________
% Text
%
\section{Introduction}
\label{sect:intro}
We introduce a general technique for designing strongly polynomial
primal-dual algorithms for a class of combinatorial optimization problems
with concave costs. We apply the technique to study three such problems: the
concave cost facility location problem, the economic lot-sizing problem with
general concave ordering costs, and the joint replenishment problem with
general concave individual ordering costs.

In the second author's Ph.D. thesis \cite{n:Stratila:2008:COP} (see also
\cite{n:Magnanti:2012:SCO}), we developed a general approach for
approximating an optimization problem with a separable concave objective by
an optimization problem with a piecewise-linear objective and the same
feasible set. When we are minimizing a nonnegative cost function over a
polyhedron, and would like the resulting problem to provide a $1+\epsilon$
approximation to the original problem in optimal cost, the size of the
resulting problem is polynomial in the size of the original problem and
linear in $1/\epsilon$. This bound implies that a variety of polynomial-time
exact algorithms, approximation algorithms, and polynomial-time heuristics
for combinatorial optimization problems immediately yield fully
polynomial-time approximation schemes, approximation algorithms, and
polynomial-time heuristics for the corresponding concave cost problems.

However, the piecewise-linear approach developed in
\cite{n:Stratila:2008:COP} cannot fully address several difficulties
involving concave cost combinatorial optimization problems. First, the
approach adds a relative error of $1+\epsilon$ in optimal cost. For example,
using the approach together with an exact algorithm for the classical
lot-sizing problem, we can obtain a fully polynomial-time approximation
scheme for the lot-sizing problem with general concave ordering costs.
However, there are exact algorithms for lot-sizing with general concave
ordering costs \cite[e.g.][]{Wagner:1960:PDP,MR1222625}, making fully
polynomial-time approximation schemes of limited interest.

Second, suppose that we are computing near-optimal solutions to a concave
cost problem by performing a $1+\epsilon$ piecewise-linear approximation,
and then using a heuristic for the resulting combinatorial optimization
problem. We are facing a trade-off between choosing a larger value of
$\epsilon$ and introducing an additional approximation error, or choosing a
smaller value of $\epsilon$ and having to solve larger combinatorial
optimization problems. For example, in \cite{n:Stratila:2008:COP}, we
computed near-optimal solutions to large-scale concave cost multicommodity
flow problems by performing piecewise-linear approximations with
$\epsilon=1\%$, and then solving the resulting fixed-charge multicommodity
flow problems with a primal-dual heuristic. The primal-dual heuristic itself
yielded an average approximation guarantee of $3.24\%$. Since we chose
$\epsilon=1\%$, the overall approximation guarantee averaged $4.27\%$. If we
were to choose $\epsilon=0.1\%$ in an effort to lower the overall guarantee,
the size of the resulting problems would increase by approximately a factor
of $10$.

Third, in some cases, after we approximate the concave cost problem by a
piecewise-linear problem, the resulting problem does not reduce polynomially
to the corresponding combinatorial optimization problem. As a result, the
piecewise-linear approach in \cite{n:Stratila:2008:COP} cannot obtain fully
polynomial-time approximation schemes, approximation algorithms, and
polynomial-time heuristics for the concave cost problem. For example, when
we carry out a piecewise-linear approximation of the joint replenishment
problem with general concave individual ordering costs, the resulting joint
replenishment problem with piecewise-linear individual ordering costs can be
reduced only to an exponentially-sized classical joint replenishment
problem.

These difficulties are inherent in any piecewise-linear approximation
approach, and cannot be addressed fully without making use of the problem
structure.

The technique developed in this paper yields a strongly polynomial
primal-dual algorithm for a concave cost problem whenever such an algorithm
exists for the corresponding combinatorial optimization problem. The
resulting algorithm runs directly on the concave cost problem, yet can be
viewed as the original algorithm running on an exponentially or
infinitely-sized combinatorial optimization problem. Therefore, exact
algorithms yield exact algorithms, and constant factor approximation ratios
are preserved. Since the execution of the resulting algorithm mirrors that
of the original algorithm, we can also expect the aposteriori approximation
guarantees of heuristics to be similar in many cases.
\subsection{Literature Review}
\label{sect:intro:lit-rev}
\subsubsection{Concave Cost Facility Location}
\label{sect:intro:loc}
In the \emph{classical facility location} problem, there are $m$ customers
and $n$ facilities. Each customer $i$ has a demand $d_i > 0$, and needs to
be connected to an open facility to satisfy this demand. Connecting a
customer $i$ to a facility $j$ incurs a connection cost $c_{ij} d_i$; we
assume that the connection costs are nonnegative and satisfy the metric
inequality. Each facility $j$ has an associated opening cost $f_j\in
\bbR_+$. Let $x_{ij} = 1$ if customer $i$ is connected to facility $j$, and
$x_{ij} = 0$ otherwise. Also let $y_j = 1$ if facility $j$ is open, and $y_j
= 0$ otherwise. Then the total cost is $\sum_{j=1}^n f_j y_j + \sum_{i=1}^m
\sum_{j=1}^n c_{ij} d_i x_{ij}$. The goal is to assign each customer to one
facility, while minimizing the total cost.

The classical facility location problem is one of the fundamental problems
in operations research \cite{MR1066259,MR2000c:90001}. The reference book
edited by Mirchandani and Francis \cite{MR1066256} introduces and reviews
the literature for a number of location problems, including classical
facility location. Since in this paper, our main contributions to facility
location problems are in the area of approximation algorithms, we next
provide a brief survey of previous approximation algorithms for classical
facility location.

Hochbaum \cite{MR643582} showed that the greedy algorithm provides a $O(\log
n)$ approximation for this problem, even when the connection costs $c_{ij}$
are not metric. Shmoys et al. \cite{Shmoys:1997:AAF} gave the first
constant-factor approximation algorithm, with a guarantee of 3.16. More
recently, Jain et al. introduced primal-dual 1.861 and 1.61-approximation
algorithms \cite{MR2146253}. Sviridenko \cite{MR2061058} obtained a
1.582-approximation algorithm based on LP rounding. Mahdian et al
\cite{MR2247734} developed a 1.52-approximation algorithm that combines a
primal-dual stage with a scaling stage. Currently, the best known ratio is
1.4991, achieved by an algorithm that employs a combination of LP rounding
and primal-dual techniques, due to Byrka \cite{Byrka:2007:OBA}.

Concerning complexity of approximation, the more general problem where the
connection costs need not be metric has the set cover problem as a special
case, and therefore is not approximable to within a certain logarithmic
factor unless $\mathrm{P} = \mathrm{NP}$ \cite{MR1715654}. The problem with
metric costs does not have a polynomial-time approximation scheme unless
$\mathrm{P} = \mathrm{NP}$, and is not approximable to within a factor of
1.463 unless $\mathrm{NP} \subseteq \mathrm{DTIME}\left(\eta^{O(\log \log
  \eta)}\right)$ \cite{MR1682440}.

A central feature of location models is the economies of scale that can be
achieved by connecting multiple customers to the same facility. The
classical facility location problem models this effect by including a fixed
charge $f_j$ for opening each facility $j$. As one of the simplest forms of
concave functions, fixed charge costs enable the model to capture the
trade-off between opening many facilities in order to decrease the
connection costs and opening few facilities to decrease the facility costs.
The \emph{concave cost facility location} problem generalizes this model by
assigning to each facility $j$ a nondecreasing concave cost function $\phi_j
: \bbR_+ \to \bbR_+$, capturing a wider variety of phenomena than is
possible with fixed charges. We assume without loss of generality that
$\phi_j(0) = 0$ for all $j$. The cost at facility $j$ is a function of the
total demand at $j$, that is $\phi_j\bigl(\sum_{i=1}^m d_i x_{ij} \bigr)$,
and the total cost is $\sum_{j=1}^n \phi_j\bigl(\sum_{i=1}^m d_i x_{ij}
\bigr) + \sum_{i=1}^m \sum_{j=1}^n c_{ij} d_i x_{ij}$.

Researchers have studied the concave cost facility location problem since at
least the 1960's \cite{Kuehn:1963:HPL,Feldman:1966:WLC}. Since it contains
classical facility location as a special case, the previously mentioned
complexity results hold for this problem---the more general non-metric
problem cannot be approximated to within a certain logarithmic factor unless
$\mathrm{P} = \mathrm{NP}$, and the metric problem cannot be approximated to
within a factor of 1.463 unless $\mathrm{NP} \subseteq
\mathrm{DTIME}\left(\eta^{O(\log \log \eta)}\right)$.

To the best of our knowledge, previously the only constant factor
approximation algorithm for concave cost facility location was obtained by
Mahdian and Pal \cite{MR2085471}, who developed a $3+\epsilon$ approximation
algorithm based on local search.

When the concave cost facility location problem has uniform demands, that is
$d_1=d_2=\dots=d_m$, a wider variety of results become available. Hajiaghayi
et al. \cite{MR1991628} obtained a 1.861-approximation algorithm. A number
of results become available due to the fact that concave cost facility
location with uniform demands can be reduced polynomially to classical
facility location. For example, Hajiaghayi et al. \cite{MR1991628} and
Mahdian et al. \cite{MR2247734} described a 1.52-approximation algorithm.

In the second author's Ph.D. thesis \cite{n:Stratila:2008:COP} (see also
\cite{n:Magnanti:2012:SCO}), we obtain a $1.4991 + \epsilon$ approximation
algorithm for concave cost facility location by using piecewise-linear
approximation. The running time of this algorithm depends polynomially on
$1/\epsilon$; when $\epsilon$ is fixed, the running time is not strongly
polynomial.

Independently, Romeijn et al. \cite{MR2598822} developed strongly polynomial
1.61 and 1.52-approximation algorithms for this problem, each with a running
time of $O(n^4\log n)$. Here $n$ is the higher of the number of customers
and the number of facilities. They consider the algorithms for classical
facility location from \cite{MR2146253,MR2247734} through a greedy
perspective. Since this paper uses a primal-dual perspective, establishing a
connection between the research of Romeijn et al. and ours is an interesting
question.
\subsubsection{Concave Cost Lot-Sizing}
\label{sect:intro:lots}
In the \emph{classical lot-sizing} problem, we have $n$ discrete time
periods, and a single item (sometimes referred to as a product, or
commodity). In each time period $t=1,\dots,n$, there is a demand $d_t\in
\bbR_+$ for the product, and this demand must be supplied from product
ordered at time $t$, or from product ordered at a time $s<t$ and held until
time $t$. In the inventory literature this requirement is known as no
backlogging and no lost sales. The cost of placing an order at time $t$
consists of a fixed cost $f_t\in \bbR_+$ and a per-unit cost $c_t\in
\bbR_+$: ordering $\xi_t$ units costs $f_t + c_t \xi_t$. Holding inventory
from time $t$ to time $t+1$ involves a per-unit holding cost $h_t \in
\bbR_+$: holding $\xi_t$ units costs $h_t \xi_t$. The goal is to satisfy all
demand, while minimizing the total ordering and holding cost.

The classical lot-sizing problem is one of the basic problems in inventory
management and was introduced by Manne \cite{Manne:1958:PEL}, and Wagner and
Whitin \cite{MR0102442}. The literature on lot-sizing is extensive and here
we provide only a brief survey of algorithmic results; for a broader
overview, the reader may refer to the book by Pochet and Wolsey
\cite{MR2219955}. Wagner and Whitin \cite{MR0102442} provided a $O(n^2)$
algorithm under the assumption that $c_t \le c_{t-1} + h_{t-1}$; this
assumption is also known as the Wagner-Whitin condition, or the
non-speculative condition. Zabel \cite{Zabel:1964:SGI}, and Eppen et al
\cite{Eppen:1969:EPH} obtained $O(n^2)$ algorithms for the general case.
Federgruen and Tzur \cite{Federgruen:1991:SFA}, Wagelmans et al.
\cite{MR1152747}, and Aggarwal and Park \cite{MR1222625} independently
obtained $O(n\log n)$ algorithms for the general case.

Krarup and Bilde \cite{MR0525912} showed that integer programming
formulation used in Section \ref{sect:lots} is integral. Levi et al.
\cite{MR2233997} also showed that this formulation is integral, and gave a
primal-dual algorithm to compute an optimal solution. (They do not evaluate
the running time of their algorithm.)

The \emph{concave cost lot-sizing} problem generalizes classical lot-sizing
by replacing the fixed and per-unit ordering costs $f_t$ and $c_t$ with
nondecreasing concave cost functions $\phi_t : \bbR_+ \to \bbR_+$. The cost
of ordering $\xi_t$ units at time $t$ is now $\phi_t(\xi_t)$. We assume
without loss of generality that $\phi_t(0) = 0$ for all $t$. This problem
has also been studied since at least the 1960's. Wagner
\cite{Wagner:1960:PDP} obtained an exact algorithm for this problem.
Aggarwal and Park \cite{MR1222625} obtain another exact algorithm with a
running time of $O(n^2)$.
\subsubsection{Concave Cost Joint Replenishment}
\label{sect:intro:jrp}
In the \emph{classical joint replenishment} problem (JRP), we have $n$
discrete time periods, and $K$ items (which may also be referred to as
products, or commodities). For each item $k$, the set-up is similar to the
classical lot-sizing problem. There is a demand $d^k_t\in \bbR_+$ of item
$k$ in time period $t$, and the demand must be satisfied from an order at
time $t$, or from inventory held from orders at times before $t$. There is a
per-unit cost $h^k_t \in \bbR_+$ for holding a unit of item $k$ from time
$t$ to $t+1$. For each order of item $k$ at time $t$, we incur a fixed cost
$f^k \in \bbR_+$. Distinguishing the classical JRP from $K$ separate
classical lot-sizing problems is the fixed joint ordering cost---for each
order at time $t$, we pay a fixed cost of $f^0\in \bbR_+$, independent of
the number of items or units ordered at time $t$. Note that $f^0$ and $f^k$
do not depend on the time $t$. The goal is to satisfy all demand, while
minimizing the total ordering and holding cost.

The classical JRP is a basic model in inventory theory
\cite{n:Joneja:1987:MEJ,Askoy:1988:MII}. The problem is NP-hard
\cite{MR995961}. When the number of items or number of time periods is
fixed, the problem can be solved in polynomial time
\cite{Zangwill:1966:DMP,MR0246626}. Federgruen and Tzur \cite{MR1307806}
developed a heuristic that computes $1+\epsilon$ approximate solutions
provided certain input parameters are bounded. Shen et al
\cite{n:Shen:X:AAS} obtained a $O(\log n + \log K)$ approximation algorithm
for the one-warehouse multi-retailer problem, which has the classical JRP as
a special case. Levi et al. \cite{MR2233997} provided the first constant
factor approximation algorithm for the classical JRP, a 2-approximation
primal-dual algorithm. Levi et al. \cite{Levi:2005:CAA} obtained a
2.398-approximation algorithm for the one-warehouse multi-retailer problem.
Levi and Sviridenko \cite{MR2305009} improved the approximation guarantee
for the one-warehouse multi-retailer problem to 1.8.

The \emph{concave cost joint replenishment} problem generalizes the
classical JRP by replacing the fixed individual ordering costs $f^k$ by
nondecreasing concave cost functions $\phi^k : \bbR_+ \to \bbR_+$. We assume
without loss of generality that $\phi^k(0) = 0$ for all $k$. The methods
employed by Zangwill \cite{Zangwill:1966:DMP} and Veinott \cite{MR0246626}
for the classical JRP with a fixed number of items or fixed number of time
periods can also be employed on the concave cost JRP. We are not aware of
results for the concave cost JRP that go beyond those available for the
classical JRP. Since prior to the work of Levi et al. \cite{MR2233997}, a
constant factor approximation algorithm for the classical JRP was not known,
we conclude that no constant factor approximation algorithms are known for
the concave cost JRP.
\subsection{Our Contribution}
\label{sect:intro:our-contr}
In Section \ref{sect:loc}, we develop our algorithm design technique for the
concave cost facility location problem. In Section \ref{sect:loc:techn}, we
describe preliminary concepts. In Sections \ref{sect:loc:single-fac} and
\ref{sect:loc:other-rules}, we obtain the key technical insights on which
our approach is based. In Section \ref{sect:loc:multi-fac}, we obtain a
strongly polynomial 1.61-approximation algorithm for concave cost facility
location with a running time of $O(m^3n + mn\log n)$. We can also obtain a
strongly polynomial 1.861-approximation algorithm with a running time of
$O(m^2n + mn\log n)$. Here $m$ denotes the number of customers and $n$ the
number of facilities.

In Section \ref{sect:lots}, we apply our technique to the concave cost
lot-sizing problem. We first adapt the algorithm of Levi et al.
\cite{MR2233997} to work for the classical lot-sizing problem as defined in
this paper. Levi et al. derive their algorithm in a slightly different
setting that is neither a generalization nor a special case of the setting
in this paper. In Section \ref{sect:lots:techn}, we obtain a strongly
polynomial exact algorithm for concave cost lot-sizing with a running time
of $O(n^2)$. Here $n$ is the number of time periods. While the running time
matches that of the fastest previous algorithm \cite{MR1222625}, our main
goal is to use this algorithm as a stepping stone in the development of our
approximation algorithm for the concave cost JRP in the following section.

In Section \ref{sect:jrp}, we apply our technique to the concave cost JRP.
We first describe the difficulty in using piecewise-linear approximation on
the concave cost JRP. We then introduce a more general version of the
classical JRP, which we call generalized JRP, and an exponentially-sized
integer programming formulation for it. In Section \ref{sect:jrp:gener},
using the 2-approximation algorithm of Levi et al. \cite{MR2233997} as the
basis, we obtain an algorithm for the generalized JRP that provides a
4-approximation guarantee and has exponential running time. In Section
\ref{sect:jrp:techn}, we obtain a strongly polynomial 4-approximation
algorithm for the concave cost JRP.
\section{Concave Cost Facility Location}
\label{sect:loc}
We first develop our technique for concave cost facility location, and then
apply it to other problems. We begin by describing the 1.61-approximation
algorithm for classical facility location due to Jain et al.
\cite{MR2146253}. We assume the reader is familiar with the primal-dual
method for approximation algorithms \cite[see e.g.][]{Goemans:1997:PDM}.

Let $[n] = \{1, \dots, n\}$. The classical facility location problem,
defined in Section \ref{sect:intro:loc}, can be formulated as an integer
program as follows:
\begin{subequations}
\label{ip:loc:classic}
\begin{align}
\min\ &\sum_{j=1}^n f_j y_j + \sum_{i=1}^m \sum_{j=1}^n c_{ij} d_i x_{ij},\\
\text{s.t.}\ &\sum_{j=1}^n x_{ij}=1, &i\in [m],\\
&0\le x_{ij} \le y_j, &i\in [m], j\in [n],\\
&y_j\in \{0,1\}, &j\in [n].
\end{align}
\end{subequations}
Recall that $f_j \in \bbR_+$ are the facility opening costs, $c_{ij} \in
\bbR_+$ are the costs of connecting customers to facilities, and $d_i \in
\bbR_+$ are the customer demands. We assume that the connection costs
$c_{ij}$ obey the metric inequality, and that the demands $d_i$ are
positive. Note that we do not need the constrains $x_{ij}\in \{0,1\}$, since
for any fixed $y\in \{0,1\}^n$, the resulting feasible polyhedron is
integral.

Consider the linear programming relaxation of problem (\ref{ip:loc:classic})
obtained by replacing the constraints $y_j\in \{0,1\}$ with $y_j\ge 0$. The
dual of this LP relaxation is:
\begin{subequations}
\label{lp:loc:classic-d}
\begin{align}
\max\ &\sum_{i=1}^m v_i,\\
\text{s.t.}\ &v_i \le c_{ij} d_i + w_{ij}, &i\in [m], j\in [n],
\label{lp:loc:classic-d:v}\\
&\sum_{i=1}^m w_{ij} \le f_j, &j\in [n],\\
&w_{ij} \ge 0, &i\in [m], j\in [n].
\end{align}
\end{subequations}
Since $w_{ij}$ do not appear in the objective, we can assume that they are
as small as possible without violating constraint
(\ref{lp:loc:classic-d:v}). In other words, we assume the invariant $w_{ij}
= \max\{0, v_i - c_{ij} d_i\}$. We will refer to dual variable $v_i$ as the
\emph{budget} of customer $i$. If $v_i \ge c_{ij} d_i$, we say that customer
$i$ \emph{contributes} to facility $j$, and $w_{ij}$ is its contribution.
The total contribution received by a facility $j$ is $\sum_{i=1}^m w_{ij}$.
A facility $j$ is \emph{tight} if $\sum_{i=1}^m w_{ij}=f_j$ and
\emph{over-tight} if $\sum_{i=1}^m w_{ij} > f_j$.

The primal complementary slackness constraints are:
\begin{subequations}
\label{eq:loc:slack}
\begin{align}
&x_{ij}(v_i - c_{ij} d_i - w_{ij}) = 0, &i\in [m], j\in [n],
\label{eq:loc:slack:x}\\
&y_j\left(\sum_{i=1}^n w_{ij} - f_j\right) = 0, &j\in [n].
\label{eq:loc:slack:y}
\end{align}
\end{subequations}
Suppose that $(x,y)$ is an integral primal feasible solution, and $(v,w)$ is
a dual feasible solution. Then, constraint (\ref{eq:loc:slack:x}) says that
customer $i$ can connect to facility $j$ (i.e. $x_{ij} = 1$) in the primal
solution only if $j$ is the closest to $i$ with respect to the modified
connection costs $c_{ij} + w_{ij} / d_i$. Constraint (\ref{eq:loc:slack:y})
says that facility $j$ can be opened in the primal solution (i.e. $y_j=1$)
only if it is tight in the dual solution.

The algorithm of Jain et al. starts with dual feasible solution $(v,w)=0$
and iteratively updates it, while maintaining dual feasibility and
increasing the dual objective. (The increase in the dual objective is not
necessarily monotonic.) At the same time, guided by the primal complementary
slackness constraints, the algorithm constructs an integral primal solution.
The algorithm concludes when the integral primal solution becomes feasible;
at this point the dual feasible solution provides a lower bound on the
optimal value.

We introduce the notion of time, and associate to each step of the algorithm
the time when it occurs. In the algorithm, we denote the time by $t$.
\begin{algorithm}[H]
\small
\algname{Algorithm FLPD}{$m, n \in \mathbb{Z}_+$; $c\in \bbR^{mn}_+, 
f\in \bbR^n_+, d\in \bbR^m_+$}
\begin{algtab}
\alglabel{alg:flpd:start}
Start at time $t=0$ with the dual solution $(v,w) = 0$. All facilities are
closed and all customers are unconnected, i.e. $(x,y)=0$.\\
\alglabel{alg:flpd:while}
\textbf{While} there are unconnected customers:\\
\algbegin 
\alglabel{alg:flpd:comp_t} 
Increase $t$ continuously. At the same
time increase $v_i$ and $w_{ij}$ for unconnected customers $i$ so as to
maintain $v_i = t d_i$ and $w_{ij} = \max \{ 0, v_i - c_{ij} d_i\}$. The
increase stops when a closed facility becomes tight, or an unconnected
customer begins contributing to an open facility.\\
\alglabel{alg:flpd:open} 
If a closed facility $j$ became tight, open it. For
each customer $i$ that contributes to $j$, connect $i$ to $j$, set $v_i =
c_{ij} d_i$, and set $w_{ij'} = \max\{0, v_i - c_{ij'} d_i\}$ for all
facilities $j'$. \\
\alglabel{alg:flpd:conn} 
If an unconnected customer $i$ began contributing
to an open facility $j$, connect $i$ to $j$.
\\ \algend
\alglabel{alg:flpd:ret}
Return $(x,y)$ and $(v,w)$.
\end{algtab}
\end{algorithm}
\vskip -1em 
In case of a tie between tight facilities in step
(\ref{alg:flpd:open}), between customers in step (\ref{alg:flpd:conn}), or
between steps (\ref{alg:flpd:open}) and (\ref{alg:flpd:conn}), we break the
tie arbitrarily. Depending on the customers that remain unconnected, in the
next iteration of loop (\ref{alg:flpd:while}), another one of the facilities
involved in the tie may open immediately, or another one of the customers
involved in the tie may connect immediately.
\begin{theorem}[\citealp{MR2146253}]
\label{th:flpd:161}
Algorithm \textsc{FLPD} is a 1.61-approximation algorithm for the classical
facility location problem.
\end{theorem}
Note that the integer program and the algorithm in our presentation are
different from those in \cite{MR2146253}. However both the integer program
and the algorithm are equivalent to those in the original presentation.
\subsection{The Technique}
\label{sect:loc:techn}
The concave cost facility location problem, also defined in Section
\ref{sect:intro:loc}, can be written as a mathematical program:
\begin{subequations}
\label{mp:loc:conc}
\begin{align}
\min\ &\sum_{j=1}^n \phi_j\left(\sum_{i=1}^m d_i x_{ij}\right) +
\sum_{i=1}^m \sum_{j=1}^n c_{ij} d_i x_{ij},\\ 
\text{s.t.}\ &\sum_{j=1}^n x_{ij}=1, &i\in [m],\\ 
&x_{ij} \ge 0, &i\in [m], j\in [n].
\end{align}
\end{subequations}
Here, $\phi_j : \bbR_+ \to \bbR_+$ are the facility cost functions, with
each function being concave nondecreasing. Assume without loss of generality
that $\phi_j(0) = 0$ for all cost functions. We omit the constraints
$x_{ij}\in \{0,1\}$, which are automatically satisfied at any vertex of the
feasible polyhedron. Since the objective is concave, this problem always has
a vertex optimal solution \cite{MR0131816}.

Suppose that the concave functions $\phi_j$ are piecewise linear on
$(0,+\infty)$ with $P$ pieces. The functions can be written as
\begin{equation}
\phi_j(\xi_j) = 
\begin{cases}
\min \{  f_{jp} + s_{jp} \xi_j : p\in [P] \}, &\xi_j > 0,\\
0, &\xi_j = 0.
\end{cases}
\end{equation}
As is well-known \cite[e.g.][]{Feldman:1966:WLC}, in this case problem
\eqref{mp:loc:conc} can be written as the following integer program:
\begin{subequations}
\label{ip:loc:pwl}
\begin{align}
\min\ &\sum_{j=1}^n \sum_{p=1}^P f_{jp} y_{jp} + \sum_{i=1}^m \sum_{j=1}^n 
\sum_{p=1}^P (c_{ij}+s_{jp}) d_i x_{ijp},\\
\text{s.t.}\ &\sum_{j=1}^n \sum_{p=1}^P x_{ijp}=1, \hskip 11em i\in [m],\\
&0\le x_{ijp} \le y_{jp}, \hskip 8.25em i\in [m], j\in [n],p\in [P],\\
&y_{jp}\in \{0,1\}, \hskip 11.125em j\in [n], p\in [P].
\end{align}
\end{subequations}
This integer program is a classical facility location problem with $Pn$
facilities and $m$ customers. Every piece $p$ in the cost function $\phi_j$
of every facility $j$ in problem \eqref{mp:loc:conc} corresponds to a
facility $\{j,p\}$ in this problem. The new facility has opening cost
$f_{jp}$. The set of customers is the same, and the connection cost from
facility $\{j,p\}$ to customer $i$ is $c_{ij} + s_{jp}$. Note that the new
connection costs again satisfy the metric inequality.

We now return to the general case, when the functions $\phi_j$ need not be
piecewise linear. Assume that $\phi_j$ are given by an oracle that returns
the function value $\phi_j(\xi_j)$ and derivative $\phi'_j(\xi_j)$ in time
$O(1)$ for $\xi_j > 0$. If the derivative at $\xi_j$ does not exist, the
oracle returns the right derivative, and we denote $\phi'_j(\xi_j) =
\lim_{\zeta \to \xi_j+} \frac{\phi_j(\zeta) - \phi_j(\xi_j)}{\zeta -
  \xi_j}$. The right derivative always exists at $\xi_j>0$, since $\phi_j$
is concave on $[0,+\infty)$.

We interpret each concave function $\phi_j$ as a piecewise-linear function
with an infinite number of pieces. For each $p>0$, we introduce a tangent
$f_{jp} + s_{jp} \xi_j$ to $\phi_j$ at $p$, with
\begin{equation}
s_{jp} = \phi'_j(p),
\qquad
f_{jp} = \phi_j(p) - p s_{jp}.
\end{equation}
We also introduce a tangent $f_{j0} + s_{j0} \xi_j$ to $\phi_j$ at $0$, with
$f_{j0} = \lim_{p\to 0+} f_{jp}$ and $s_{j0} = \lim_{p\to 0+} s_{jp}$. The
limit $\lim_{p\to 0+} f_{jp}$ is finite because $f_{jp}$ are nondecreasing
in $p$ and bounded from below. The limit $\lim_{p\to 0+} s_{jp}$ is either
finite or $+\infty$ because $s_{jp}$ are nonincreasing in $p$, and we assume
that this limit is finite.

Our technique also applies when $\lim_{p\to 0+} s_{jp} = +\infty$, in which
case we introduce tangents to $\phi_j$ only at points $p>0$, and then
proceed in similar fashion. In some computational settings, using
derivatives is computationally expensive. In such cases, we can assume that
the demands are rational, and let $d_i = \frac{d'_i}{d''_i}$ with $d''_i>0$,
and $d'_i$ and $d''_i$ coprime integers. Also let $\Delta = \frac{1}{d''_1
  d''_2 \dots d''_m}$. Then, we can use the quantity $\frac{\phi_j(\lfloor
  \xi_j+\Delta \rfloor) - \phi_j(\lfloor \xi_j \rfloor)}{\Delta}$ instead of
$\phi'_j(\xi_j)$ throughout.

The functions $\phi_j$ can now be expressed as:
\begin{equation}
\label{eq:loc:pwl-inf}
\phi_j(\xi_j) = \begin{cases}
\min \{ f_{jp} + s_{jp} \xi_j : p\ge 0\}, &\xi_j > 0,\\
0, &\xi_j=0.
\end{cases}
\end{equation}
When $\phi_j$ is linear on an interval $[\zeta_1, \zeta_2]$, all points
$p\in [\zeta_1, \zeta_2)$ yield the same tangent, that is $(f_{jp}, s_{jp})
  = (f_{jq}, s_{jq})$ for any $p,q\in [\zeta_1, \zeta_2)$. For convenience,
    we consider the tangents $(f_{jp}, s_{jp})$ for all $p\ge 0$, regardless
    of the shape of $\phi_j$. Sometimes, we will refer to a tangent
    $(f_{jp}, s_{jp})$ by the point $p$ that gave rise to it.

We apply formulation \eqref{ip:loc:pwl}, and obtain a classical facility
location problem with $m$ customers and an infinite number of facilities.
Each tangent $p$ to cost function $\phi_j$ of facility $j$ in problem
\eqref{mp:loc:conc} corresponds to a facility $\{j,p\}$ in the resulting
problem. Due to their origin, we will sometimes refer to facilities in the
resulting problem as tangents.

The resulting integer program is:
\begin{subequations}
\label{ip:loc:pwl-inf}
\begin{align}
\min\ &\sum_{j=1}^n \sum_{p\ge 0} f_{jp} y_{jp} + \sum_{i=1}^m 
\sum_{j=1}^n \sum_{p\ge 0} (c_{ij}+s_{jp}) d_i x_{ijp},\\
\text{s.t.}\ &\sum_{j=1}^n \sum_{p\ge 0} x_{ijp}=1, 
\hskip 10em i\in [m],\\
&0\le x_{ijp} \le y_{jp}, \hskip 8em i\in [m], j\in [n],p\ge 0,\\
&y_{jp}\in \{0,1\}, \hskip 10.75 em j\in [n], p\ge 0.
\end{align}
\end{subequations}
Of course, we cannot run Algorithm \textsc{FLPD} on this problem directly,
as it is infinitely-sized. Instead, we will show how to execute Algorithm
\textsc{FLPD} on this problem implicitly. Formally, we will devise an
algorithm that takes problem \eqref{mp:loc:conc} as input, runs in
polynomial time, and produces the same assignment of customers to facilities
as if Algorithm \textsc{FLPD} were run on problem \eqref{ip:loc:pwl-inf}.
Thereby, we will obtain a 1.61-approximation algorithm for problem
\eqref{mp:loc:conc}. We will call the new algorithm \textsc{ConcaveFLPD}.

The LP relaxation of problem \eqref{ip:loc:pwl-inf} is obtained by replacing
the constraints $y_{jp} \in \{0, 1\}$ with $y_{jp}\ge 0$. The dual of the LP
relaxation is:
\begin{subequations}
\label{lp:loc:pwl-inf-d}
\begin{align}
\max\ &\sum_{i=1}^m v_i,\\
\text{s.t.}\ &v_i \le (c_{ij} + s_{jp}) d_i + w_{ijp}, &i\in [m], 
j\in [n], p\ge 0, \label{lp:loc:pwl-inf-d:v}\\
&\sum_{i=1}^m w_{ijp} \le f_{jp}, &j\in [n], p\ge 0,\\
&w_{ijp} \ge 0, &i\in [m], j\in [n], p\ge 0.
\end{align}
\end{subequations}
Since the LP relaxation and its dual are infinitely-sized, the strong
duality property does not hold automatically, as in the finite LP case.
However, we do not need strong duality for our approach. We rely only on the
fact that the optimal value of integer program (\ref{ip:loc:pwl-inf}) is at
least that of its LP relaxation, and on weak duality between the LP
relaxation and its dual.
\subsection{Analysis of a Single Facility}
\label{sect:loc:single-fac}
In this section, we prove several key lemmas that will enable us to execute
Algorithm \textsc{FLPD} implicitly. We prove the lemmas in a simplified
setting when problem \eqref{mp:loc:conc} has only one facility, and all
connection costs $c_{ij}$ are zero. To simplify the notation, we omit the
facility subscript $j$.

Imagine that we are at the beginning of step (\ref{alg:flpd:comp_t}) of the
algorithm. To execute this step, we need to compute the time when the
increase in the dual variables stops. The increase may stop because a closed
tangent became tight, or because an unconnected customer began contributing
to an open tangent. We assume that there are no open tangents, which implies
that the increase stops because a closed tangent became tight.

Let $t=0$ at the beginning of the step, and imagine that $t$ is increasing
to $+\infty$. The customer budgets start at $v_i \ge 0$ and increase over
time at rates $\delta_i$. At time $t$, the budget for customer $i$ has
increased to $v_i + t \delta_i$. Connected customers are modeled by taking
$\delta_i = 0$, and unconnected customers by taking $\delta_i = d_i$. Denote
the set of connected customers by $C$ and the set of unconnected customers
by $U$, and let $\mu=|U|$.

First, we consider the case when all customers have zero starting budgets.
\begin{lemma}
\label{lm:loc:uni-sim}
If $v_i = 0$ for $i\in [m]$, then tangent $p^* = \sum_{i\in U} d_i$ becomes
tight first, at time $t^* = s_{p^*}+ \frac{f_{p^*}}{p^*}$. If there is a
tie, it is between at most two tangents.
\end{lemma}
\begin{proof}
A given tangent $p$ becomes tight at time $s_p+\frac{f_p}{\sum_{i\in U}
  d_i}$. Therefore,
\begin{equation}
p^* = \argmin_{p\ge 0} \left\{ s_p+\frac{f_p}{\sum_{i\in U} d_i}\right\}
= \argmin_{p\ge 0} \left\{ s_p \sum_{i\in U} d_i + f_p\right\}.
\end{equation}
The quantity $s_p \sum_{i\in U} d_i + f_p$ can be viewed as the value of the
affine function $f_p+s_p\xi$ at $\xi=\sum_{i\in U} d_i$. Since $f_p + s_p
\xi$ is tangent to $\phi$, and $\phi$ is concave,
\begin{equation}
f_p + s_p \sum_{i\in U} d_i \ge \phi\left(\sum_{i\in U} d_i\right)
\quad\text{for}\quad p\ge 0.
\end{equation}
On the other hand, for tangent $p^* = \sum_{i\in U} d_i$, we have $f_{p^*} +
s_{p^*} \sum_{i\in U} d_i = \phi\bigl( \sum_{i\in U} d_i \bigr)$. Therefore,
tangent $p^*$ becomes tight first, at time $t^* = s_{p^*}+
\frac{f_{p^*}}{p^*}$. (See Figure \ref{fig:loc:uni-sim}.)

Concerning ties, for a tangent $p$ to become tight first, it has to satisfy
$f_p + s_p \sum_{i\in U} d_i = \phi\bigl( \sum_{i\in U} d_i \bigr)$, or in
other words it has to be tangent to $\phi$ at $\sum_{i\in U} d_i$. We
consider two cases. First, let $\zeta_2$ be as large as possible so that
$\phi$ is linear on $[p^*, \zeta_2]$. Then, any point $p\in [p^*, \zeta_2)$
  yields the same tangent as $p^*$, that is $(f_p, s_p) = (f_{p^*},
  s_{p^*})$. Second, let $\zeta_1$ be as small as possible so that $\phi$ is
  linear on $[\zeta_1, p^*]$. Then, any point $p\in [\zeta_1, p^*)$ yields
    the same tangent as $\zeta_1$, that is $(f_p, s_p) = (f_{\zeta_1},
    s_{\zeta_1})$. Tangent $\zeta_1$ is also tangent to $\phi$ at
    $\sum_{i\in U} d_i$, and may be different from tangent $p^*$. Tangents
    $p\not\in [\zeta_1, \zeta_2]$ have $f_p + s_p \sum_{i\in U} d_i >
    \phi\bigl(\sum_{i\in U} d_i\bigr)$. Therefore, in a tie, at most two
    tangents, $\zeta_1$ and $p^*$, become tight first.
\end{proof}
\begin{figure}[t]
\begin{center}
\epsfig{file=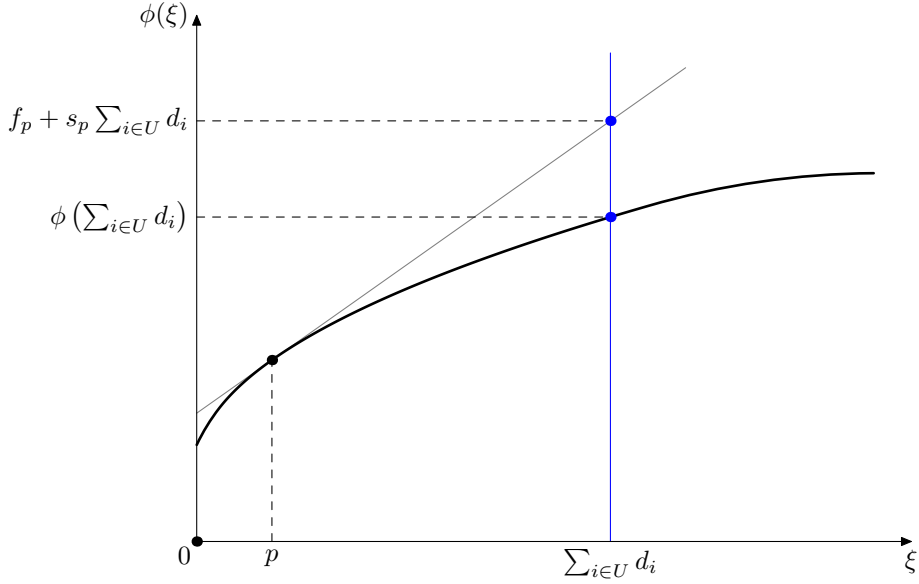}
\end{center}
\caption{Illustration of the proof of Lemma \ref{lm:loc:uni-sim}.
\label{fig:loc:uni-sim}}
\end{figure}

Next, we return to the more general case when customers have nonnegative
starting budgets. Define
\begin{equation}
\label{eq:loc:def-pi}
p_i(t) =\min\{ p\ge 0 : v_i + t \delta_i \ge s_p d_i\}, 
\qquad i\in [m],\\
\end{equation}
If $v_i + t \delta_i < s_p d_i$ for every $p\ge 0$, let $p_i(t) = +\infty$.
Otherwise, the minimum is well-defined, since $s_p$ is right-continuous in
$p$.

Intuitively, $p_i(t)$ is the leftmost tangent to which customer $i$ is
contributing at time $t$. Note that $s_p$ is decreasing in $p$, since $\phi$
is a concave function. Therefore, customer $i$ contributes to every tangent
to the right of $p_i(t)$, and does not contribute to any tangent to the left
of $p_i(t)$. For any two customers $i$ and $j$,
\begin{subequations}
\label{eq:loc:monot-pi}
\begin{align}
(v_i + t \delta_i) / d_i > (v_j + t \delta_j) / d_j 
\quad &\Rightarrow \quad p_i(t) \le p_j(t),\\
(v_i + t \delta_i) / d_i = (v_j + t \delta_j) / d_j 
\quad &\Rightarrow \quad p_i(t) = p_j(t),\\
(v_i + t \delta_i) / d_i < (v_j + t \delta_j) / d_j
\quad &\Rightarrow \quad p_i(t) \ge p_j(t).
\end{align}
\end{subequations}

Assume without loss of generality that the set of customers is ordered so
that customers $1, \dots, \mu$ are unconnected, customers $\mu+1, \dots, m$
are connected, and
\begin{subequations}
\begin{gather}
v_1 / d_1  \ge v_2 / d_2 \ge \dots \ge v_{\mu} / d_{\mu},\\
v_{\mu+1} / d_{\mu+1} \ge v_{\mu+2} / d_{\mu+2} \ge \dots \ge 
v_m / d_m. 
\end{gather}
\end{subequations}
Note that $(v_i + t \delta_i) / d_i = v_i / d_i$ for connected customers,
and $(v_i + t \delta_i) / d_i = v_i / d_i + t$ for unconnected ones. By
property (\ref{eq:loc:monot-pi}), at all times $t$, we have $p_1(t) \le
p_2(t) \le \dots \le p_{\mu}(t)$ and $p_{\mu+1}(t) \le p_{\mu+2}(t) \le
\dots \le p_m(t)$. As $t$ increases, $p_i(t)$ for $i\in C$ are unchanged,
while $p_i(t)$ for $i\in U$ decrease. (See Figure \ref{fig:loc:multi-sim}.)
\begin{figure}[t]
\begin{center}
\epsfig{file=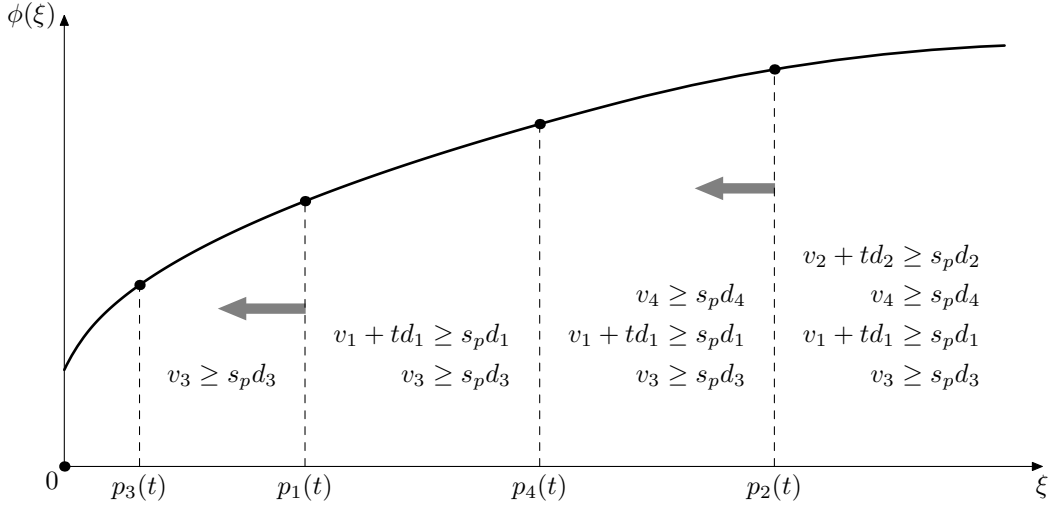}
\end{center}
\caption{Illustration of the definition of $p_i(t)$. Here $U=\{1,2\}$ and
  $C=\{3,4\}$. The gray arrows show how $p_i(t)$ change as $t$ increases.
  The inequalities show the set of customers that contribute to the tangents
  in each of the intervals defined by $p_i(t)$.
\label{fig:loc:multi-sim}}
\end{figure}

Let
\begin{subequations}
\begin{align}
&I_k^{\rmu}(t) = [p_k(t), p_{k+1}(t)), &1\le k < \mu,\\
&I_l^{\rmc}(t) = [p_l(t), p_{l+1}(t)), &\mu+1 \le l < m,
\end{align}
\end{subequations}
with $I_0^{\rmu}(t) = [0, p_1(t))$ and $I_{\mu}^{\rmu}(t) = [p_{\mu}(t),
    +\infty)$, as well as $I_{\mu}^{\rmc}(t) = [0, p_{\mu+1}(t))$ and
      $I^{\rmc}_{m}(t) = [p_m(t), +\infty)$. When an interval has the form
        $[+\infty, +\infty)$, we interpret it to be empty. Consider the
          intervals
\begin{equation}
I_{kl}(t) = I^{\rmu}_k(t) \cap I^{\rmc}_l(t), \qquad\qquad
0 \le k \le \mu \le l \le m.
\end{equation}
At any given time $t$, some of the intervals $I_{kl}(t)$ may be empty. As
time increases, these intervals may vary in size, empty intervals may become
non-empty, and non-empty intervals may become empty. The intervals partition
$[0,+\infty)$, that is $\cup_{0\le k\le \mu\le l\le m} I_{kl}(t) =
  [0,+\infty)$, and $I_{kl}(t) \cap I_{rs}(t) = \emptyset$ for $(k,l) \neq
    (r,s)$.

Let $\omega_p(t)$ be the total contribution received by tangent $p$ at time
$t$. The tangents on each interval $I_{kl}(t)$ receive contributions from
unconnected customers $\{1, \dots, k\}$ and connected customers $\{ \mu+1,
\dots, l\}$. We define $C(k,l) = \{1,\dots,k\} \cup \{\mu+1,\dots,l\}$ to be
the set of customers that contribute to tangents in $I_{kl}(t)$.

For each interval $I_{kl}(t)$ with $k\ge 1$, we define an alternate setting
$A(k,l)$, where all starting budgets are zero, customers in $C(k,l)$
increase their budgets at rates $d_i$, and the remaining customers do not
change their budgets. Let $\omega^{kl}_p(\tau_{kl})$ be the total
contribution received by tangent $p$ at time $\tau_{kl}$ in the alternate
setting $A(k,l)$. We establish a correspondence between times $t$ in the
original setting and times $\tau_{kl}$ in $A(k,l)$, given by $\tau_{kl} =
\beta_{kl} + \alpha_{kl} t$, with $\alpha_{kl} = \sum_{i=1}^k d_i \big/
\sum_{i\in C(k,l)} d_i$ and $\beta_{kl} = \sum_{i\in C(k,l)} v_i \big/
\sum_{i\in C(k,l)} d_i$. Since $\alpha_{kl}>0$, times $t\in [0,+\infty)$ are
  mapped one-to-one to times $\tau_{kl} \in [\beta_{kl}, +\infty)$.

The following two lemmas relate the original setting to the alternate
settings $A(k,l)$.
\begin{lemma}
\label{lm:loc:contrib-in}
Given a time $t$ and an interval $I_{kl}(t)$ with $k\ge 1$, any tangent
$p\in I_{kl}(t)$ receives the same total contribution at time $t$ in the
original setting as at time $\tau_{kl}$ in $A(k,l)$, that is $\omega_p(t) =
\omega^{kl}_p(\tau_{kl})$.
\end{lemma}
\begin{proof}
The total contribution to $p$ at time $t$ in the original setting is
\begin{equation}
\begin{split}
\omega_p(t) &= \sum_{i=1}^k (v_i + t d_i - s_pd_i) 
+ \sum_{i=\mu+1}^l (v_i - s_pd_i) \\
&= \sum_{i\in C(k,l)} (v_i - s_p d_i) + t\sum_{i=1}^k d_i 
= \sum_{i\in C(k,l)} (v_i + \alpha_{kl} t d_i - s_p d_i) \\
&= (\beta_{kl} + \alpha_{kl} t - s_p) \sum_{i\in C(k,l)} d_i 
= (\tau_{kl} - s_p) \sum_{i\in C(k,l)} d_i.
\end{split}
\end{equation}
Since $\omega_p(t) \ge 0$, it follows that $\tau_{kl} - s_p \ge 0$, and
therefore
\begin{equation}
(\tau_{kl} - s_p) \sum_{i\in C(k,l)} d_i = \sum_{i\in C(k,l)} \max\{ 0,
  \tau_{kl} d_i - s_p d_i \} = \omega^{kl}_p(\tau_{kl}). 
\qedhere
\end{equation}
\end{proof}
\begin{lemma}
\label{lm:loc:contrib-out}
Given a time $t$ and an interval $I_{kl}(t)$ with $k\ge 1$, a tangent $p$
receives at least as large a total contribution at time $t$ in the original
setting as at time $\tau_{kl}$ in $A(k,l)$, that is $\omega_p(t) \ge
\omega^{kl}_p(\tau_{kl})$.
\end{lemma}
\begin{proof}
If $\tau_{kl} - s_p < 0$, then $\omega^{kl}_p(\tau_{kl}) = 0$, and thus
$\omega_p(t) \ge \omega^{kl}_p(\tau_{kl})$. If $\tau_{kl} - s_p\ge 0$, then
$\omega^{kl}_p(\tau_{kl}) = (\tau_{kl} - s_p) \sum_{i\in C(k,l)} d_i =
\sum_{i=1}^k (v_i + t d_i - s_p d_i) + \sum_{i=\mu+1}^l (v_i - s_p d_i)$.
Let $p\in I_{rs}(t)$ for some $r$ and $s$, and note that $\omega_p(t) =
\sum_{i=1}^r (v_i + t d_i - s_p d_i) + \sum_{i=\mu+1}^s (v_i - s_p d_i)$.

The difference between the two contributions can be written as
\begin{multline}
\qquad \omega_p(t) - \omega^{kl}_p(\tau_{kl}) = \sum_{i=k+1}^r (v_i + t d_i
 - s_p d_i) - \sum_{i=r+1}^k (v_i + t d_i - s_p d_i) \\ 
+ \sum_{i=l+1}^s (v_i - s_p d_i) - \sum_{i=s+1}^l (v_i - s_p d_i). \qquad
\end{multline}
Note that at least two of the four summations in this expression are always
empty. We now examine the summations one by one:
\begin{enumerate}
\item
$\sum_{i=k+1}^r (v_i + t d_i - s_p d_i)$. This summation is nonempty when
  $r>k$. In this case, in the original setting, customers $k+1, \dots, r$ do
  contribute to tangents in $I_{rs}(t)$ at time $t$. Therefore, $v_i + t d_i
  - s_p d_i \ge 0$ for $i=k+1, \dots, r$, and the summation is nonnegative.
\item
$-\sum_{i=r+1}^k (v_i + t d_i - s_p d_i)$. This summation is nonempty when
  $r<k$. In this case, in the original setting, customers $r+1, \dots, k$ do
  not contribute to tangents in $I_{rs}(t)$ at time $t$. Therefore, $v_i + t
  d_i - s_p d_i \le 0$ for $i=r+1,\dots, k$, and the summation is
  nonnegative.
\item
$\sum_{i=l+1}^s (v_i - s_p d_i)$. This summation is nonempty when $s>l$. In
  this case, in the original setting, customers $l+1, \dots, s$ do
  contribute to tangents in $I_{rs}(t)$ at time $t$. Therefore, $v_i - s_p
  d_i \ge 0$ for $i=l+1, \dots, s$, and the summation is nonnegative.
\item
$-\sum_{i=s+1}^l (v_i - s_p d_i)$. This summation is nonempty when $s<l$. In
  this case, in the original setting, customers $s+1, \dots, l$ do not
  contribute to tangents in $I_{rs}(t)$ at time $t$. Therefore, $v_i - s_p
  d_i \le 0$ for $i=s+1, \dots, l$, and the summation is nonnegative.
\end{enumerate}
As a result of the above cases, we obtain that $\omega_p(t) -
\omega_p^{kl}(\tau_{kl}) \ge 0$.
\end{proof}

When $k\ge 1$, we can apply Lemma \ref{lm:loc:uni-sim} to compute the first
tangent to become tight in $A(k,l)$, and the time when this occurs. Denote
the computed tangent and time by $p'_{kl}$ and $\tau'_{kl}$, and note that
$p'_{kl} = \sum_{i\in C(k,l)} d_i$ and $\tau'_{kl} = s_{p'_{kl}} +
\frac{f_{p'_{kl}}}{p'_{kl}}$. Let $t'_{kl} = \frac{\tau'_{kl} -
  \beta_{kl}}{\alpha_{kl}}$ be the time in the original setting that
corresponds to time $\tau'_{kl}$ in $A(k,l)$. Let $t^* = \min\{ t'_{kl} :
1\le k\le \mu\le l\le m\}$, and $p^* = p'_{\argmin\{ t'_{kl} : 1\le k\le
  \mu\le l\le m\}}$. The following two lemmas will enable us to show that
tangent $p^*$ becomes tight first in the original setting, at time $t^*$.
\begin{lemma}
\label{lm:loc:tang-ge}
If a tangent $p$ becomes tight at a time $t$ in the original setting, then
$t\ge t^*$.
\end{lemma}
\begin{proof}
Since $\cup_{0\le r\le \mu\le s\le m} I_{rs}(t) = [0,+\infty)$, there is an
  interval $I_{kl}(t)$ that contains $p$. Since the contributions to
  tangents in the interval $I_{0l}(t)$ do not increase over time, $k\ge 1$.

Tangent $p$ is tight at time $t$ in the original setting, and therefore
$\omega_p(t) = f_p$. By Lemma \ref{lm:loc:contrib-in}, $p\in I_{kl}(t)$
implies that $\omega_p^{kl}(\tau_{kl}) = f_p$, and hence $p$ is tight at
time $\tau_{kl}$ in $A(k,l)$. It follows that $\tau_{kl} \ge \tau'_{kl}$,
and therefore $t \ge t'_{kl} \ge t^*$.
\end{proof}
\begin{lemma}
\label{lm:loc:tang-le}
Each tangent $p'_{kl}$ with $k\ge 1$ becomes tight at a time $t\le t'_{kl}$
in the original setting.
\end{lemma}
\begin{proof}
Since tangent $p'_{kl}$ is tight at time $\tau'_{kl}$ in $A(k,l)$, we have
$\omega_{p'_{kl}}^{kl}(\tau'_{kl}) = f_{p'_{kl}}$. By Lemma
\ref{lm:loc:contrib-out}, $\omega_{p'_{kl}}(t'_{kl}) \ge f_{p'_{kl}}$, which
means that $p'_{kl}$ is tight or over-tight at time $t'_{kl}$ in the
original setting. Therefore, $p'_{kl}$ becomes tight at a time $t\le
t'_{kl}$ in the original setting.
\end{proof}

We now obtain the main result of this section.
\begin{lemma}
\label{lm:loc:multi-sim}
Tangent $p^*$ becomes tight first in the original setting, at time $t^*$.
The quantities $p^*$ and $t^*$ can be computed in time $O(m^2)$.
\end{lemma}
\begin{proof}
Lemma \ref{lm:loc:tang-ge} implies that tangents become tight only at times
$t\ge t^*$. Lemma \ref{lm:loc:tang-le} implies that tangent $p^* =
p'_{\argmin\{ t'_{kl} : 1 \le k \le \mu \le l \le m\}}$ becomes tight at a
time $t\le \min\{ t'_{kl} : 1 \le k \le \mu \le l \le m\} = t^*$. Therefore,
tangent $p^*$ becomes tight first, at time $t^*$.

To evaluate the running time, note that the $d_i$ and $v_i$ can be sorted in
$O(m\log m)$ time. Once the $d_i$ and $v_i$ are sorted, we can compute all
quantities $\alpha_{kl}$ and $\beta_{kl}$ in $O(m^2)$, and then compute all
$t'_{kl}$ and $p'_{kl}$ in $O(m^2)$ via Lemma \ref{lm:loc:uni-sim}.
Therefore, the total running time is $O(m^2)$.
\end{proof}

In case of a tie, Lemma \ref{lm:loc:multi-sim} enables us to compute one of
the tangents that become tight first. It is possible to obtain additional
results about ties, starting with that of Lemma \ref{lm:loc:uni-sim}.
However, we do not need such results in this paper, as Algorithm
\textsc{FLPD}, as well as the algorithms in Sections \ref{sect:lots} and
\ref{sect:jrp}, allow us to break ties arbitrarily.

For many primal-dual algorithms, we can perform the computation in Lemma
\ref{lm:loc:multi-sim} faster than in $O(m^2)$, by taking into account the
details of how the algorithm increases the dual variables. We will
illustrate this with three algorithms in Sections \ref{sect:loc:multi-fac},
\ref{sect:lots}, and \ref{sect:jrp}.
\subsection{Other Rules for Changing the Dual Variables}
\label{sect:loc:other-rules}
In this section, we consider the same setting as in the previous one, but in
addition allow each customer $i$ to change its budget at an arbitrary rate
$\delta_i \ge 0$. The rate is no longer limited to the set $\{0, d_i\}$, and
we assume that at least one customer has $\delta_i > 0$. The following
results are not needed to obtain the algorithms in this paper. We include
them since they embody a more general version of our approach, and may be
useful in developing primal-dual algorithms in the future.

Consider the quantities $p_i(t)$ as defined in equation
\eqref{eq:loc:def-pi}. Since $\delta_i$ need not equal $d_i$, the order of
the $p_i(t)$ may change as $t$ increases from $0$ to $+\infty$. At any given
time $t$, the $p_i(t)$ divide $[0,+\infty)$ into at most $m+1$ intervals.
  For each set of customers $K \subseteq [m]$, we introduce an interval
\begin{equation}
I_K(t) = [a_K(t), b_K(t)) = \Bigl[ \max_{i\in K} p_i(t), \min_{i\not\in K}
    p_i(t) \Bigr).
\end{equation}
If $K=\emptyset$, we set $a_K(t) = 0$, and if $K=[m]$, we set $b_K(t) =
+\infty$. If $a_K(t) \ge b_K(t)$ or $a_K(t) = b_K(t) = +\infty$, we take
$I_K(t)$ to be empty. Note that $\bigcup_{K\subseteq [m]} I_K(t) = [0,
  +\infty)$, and $I_K(t) \cap I_L(t) = \emptyset$ for $K\neq L$. Any
  interval that is formed by the $p_i(t)$ as $t$ increases from $0$ to
  $+\infty$ is among the intervals $I_K(t)$. The set of customers
  contributing to tangents on an interval $I_K(t)$ is precisely $K$.

As in the previous section, for each interval $I_K(t)$ with $\sum_{i\in K}
\delta_i > 0$, we define an alternate setting $A(K)$, where all starting
budgets are zero, customers in $K$ increase their budgets at rates $d_i$,
and the remaining customers keep their budgets unchanged. We denote the
total contribution received by tangent $p$ at time $\tau_K$ in $A(K)$ by
$\omega_p^K(\tau_K)$. The correspondence between times $t$ in the original
setting and times $\tau_K$ in $A(K)$ is given by $\tau_K = \beta_K +
\alpha_K t$, with $\alpha_K = \sum_{i\in K} \delta_i / \sum_{i\in K} d_i$
and $\beta_K = \sum_{i\in K} v_i / \sum_{i\in K} d_i$. Since $\alpha_K > 0$,
the correspondence is one-to-one between times $t\in [0, +\infty)$ and
  $\tau_K \in [\beta_K, +\infty)$.
\begin{lemma}
\label{lm:loc:oth-r:contrib-in}
Given a time $t$ and an interval $I_K(t)$ with $\sum_{i\in K} \delta_i > 0$,
a tangent $p\in I_K(t)$ receives the same total contribution at time $t$ in
the original setting as at time $\tau_K$ in $A(K)$, that is $\omega_p(t) =
\omega_p^K(\tau_K)$.
\end{lemma}
\begin{proof}
The total contribution in the original setting is
\begin{multline}
\qquad \omega_p(t) = \sum_{i\in K} (v_i + t\delta_i - s_pd_i) 
= \sum_{i\in K} (v_i + \alpha_K t d_i - s_p d_i) \\
= (\beta_K + \alpha_K t - s_p) \sum_{i\in K} d_i 
= (\tau_K - s_p) \sum_{i\in K} d_i = \omega_p^K(\tau_K). \qquad
\addtocounter{equation}{1}
\tag*{$\begin{matrix} (\theequation)\\ \qedhere \end{matrix}$}
\end{multline}
\end{proof}
\begin{lemma}
\label{lm:loc:oth-r:contrib-out}
Given a time $t$ and an interval $I_K(t)$ with $\sum_{i\in K} \delta_i > 0$,
a tangent $p$ receives at least as large a total contribution at time $t$ in
the original setting as at time $\tau_K$ in $A(K)$, that is $\omega_p(t) \ge
\omega_p^K(\tau_K)$.
\end{lemma}
\begin{proof}
If $\tau_K - s_p < 0$, then $\omega_p(t) \ge \omega_p^K(\tau_K)$. If $\tau_K
- s_p\ge 0$, let $p\in I_L(t)$ for some $L\subseteq [m]$, and note that
$\omega_p^K(\tau_K) = (\tau_K - s_p) \sum_{i\in K} d_i = \sum_{i\in K} (v_i
+ t\delta_i - s_p d_i)$, while $\omega_p(t) = \sum_{i\in L} (v_i + t\delta_i
- s_p d_i)$.

The difference between the two contributions is
\begin{equation}
\omega_p(t) - \omega_p^K(\tau_K) = \sum_{i\in L\setminus K} (v_i + t\delta_i
- s_p d_i) - \sum_{i\in K\setminus L} (v_i + t\delta_i - s_p d_i).
\end{equation}
Since $p\in I_L(t)$, in the original setting, customers in $L$ contribute to
tangent $p$ at time $t$, and therefore $v_i + t \delta_i - s_p d_i \ge 0$
for $i\in L$, which implies that $\sum_{i\in L\setminus K} (v_i + t\delta_i
- s_p d_i) \ge 0$. Conversely, customers not in $L$ do not contribute to $p$
at time $t$, implying that $v_i + t\delta_i - s_p d_i \le 0$ for $i\not\in
L$, and therefore $\sum_{i\in K\setminus L} (v_i + t\delta_i s_p d_i) \le
0$. As a result, $\omega_p(t) - \omega_p^K(\tau_K) \ge 0$.
\end{proof}
Unlike in the previous section, we have exponentially many alternative
settings $A(K)$. The following derivations will enable us to compute the
first tangent to become tight in the original setting, and the time when
this occurs using only a polynomial number of alternative settings. 

As $t$ increases from $0$ to $+\infty$, the order of the quantities $(v_i +
t\delta_i)/d_i$ may change. Since the quantities are linear in $t$, as
$t\to+\infty$, they assume an order that no longer changes. We use this
order to define a permutation $\pi(+\infty) = (\pi_1(+\infty), \dots,
\pi_m(+\infty))$, with $\pi_i(+\infty) = j$ meaning that $(v_j +
t\delta_j)d_j$ is the $i$-th largest quantity. If two quantities are tied as
$t\to+\infty$, we break the tie arbitrarily. Similarly, for any time $t\in
[0,+\infty)$, we define a permutation $\pi(t) = (\pi_1(t), \dots,
  \pi_m(t))$. In this case, if two quantities are tied, we break the tie
  according to $\pi(+\infty)$. For example, suppose that the two largest
  quantities at time $t$ are tied, that they are $(v_1 + t\delta_1)/d_1 =
  (v_2 + t\delta_2)/d_2$, and that $\pi_i(+\infty) = 1$ and $\pi_j(+\infty)
  = 2$ with $i < j$. Then we take $\pi_1(t) = 1$ and $\pi_2(t) = 2$.

Compare two quantities
\begin{equation}
(v_i + t \delta_i) / d_i \quad \text{vs.} \quad (v_j + t \delta_j) / d_j.
\end{equation}
If their order changes as $t$ increases from $0$ to $+\infty$, then there is
a $\theta>0$ such that the sign between the quantities is `$<$' on
$[0,\theta)$, `$=$' at $\theta$, and `$>$' on $(\theta,+\infty)$, or
  vice-versa. Let $\theta_1 < \dots < \theta_R$ be all such times when the
  sign between two quantities changes, and let $\theta_0 = 0$ and
  $\theta_{R+1}=+\infty$. Since there are $m(m-1)/2$ pairs of quantities,
  $R\le m(m-1)/2$.

The proof of the following lemma follows from these definitions.
\begin{lemma}
As $t$ increases from $0$ to $+\infty$, the permutation $\pi(t)$ changes at
times $\theta_1, \dots, \theta_R$. Moreover, $\pi(t)$ is unchanged on the
intervals $[\theta_r, \theta_{r+1})$ for $r=0,\dots,R$.
\end{lemma}
We now bound the number of intervals $I_K(t)$ that ever become nonempty. Let
$\mathcal{K}(t) = \{\{ \pi_1(t), \dots, \pi_i(t) \} : i=0,\dots,m\}$ and
$\mathcal{K} = \cup_{r=0}^R \mathcal{K}(\theta_r)$, and note that
$|\mathcal{K}(t)| \le m+1$ and $|\mathcal{K}| \le (m + 1)(m(m - 1)/2 + 1) =
O(m^3)$.
\begin{lemma}
As $t$ increases from $0$ to $+\infty$, only intervals $I_K(t)$ with $K\in
\mathcal{K}$ ever become nonempty, that is $\{ K : \exists t\ge 0 \text{ \rm
  s.t. } I_K(t)\neq \emptyset\} \subseteq \mathcal{K}$.
\end{lemma}
\begin{proof}
Fix a time $t$, and note that by property \eqref{eq:loc:monot-pi}, we have
$p_{\pi_1(t)}(t) \le p_{\pi_2(t)}(t) \le \dots \le p_{\pi_m(t)}(t)$.
Therefore, the intervals $I_K(t)$ may be nonempty only when $K \in
\mathcal{K}(t)$. Since $\pi(t)$ is unchanged on the intervals $[\theta_r,
  \theta_{r+1})$ for $r=0,\dots,R$, if an interval $I_K(t)$ ever becomes
  nonempty, then $K\in \mathcal{K}$.
\end{proof}
As in the previous section, when $\sum_{i\in K} \delta_i > 0$, we can
compute the first tangent to become tight in $A(K)$, and the time when this
occurs using Lemma \ref{lm:loc:uni-sim}. Let the computed tangent and time
be $p'_K = \sum_{i\in K} d_i$ and $\tau'_K = s_{p'_K} +
\frac{f_{p'_K}}{p'_K}$, and let $t'_K = \frac{\tau'_K - \beta_K}{\alpha_K}$
be the time in the original setting corresponding to time $\tau'_K$ in
$A(K)$. Next, we show that tangent $p^* = p'_{\argmin\left\{t'_K :
  \sum_{i\in K} \delta_i > 0, K\in \mathcal{K} \right\}}$ becomes tight
first, at time $t^* = \min\bigl\{ t'_K : \sum_{i\in K} \delta_i > 0, K\in
\mathcal{K} \bigr\}$.
\begin{lemma}
\label{lm:loc:oth-r:tang-ge}
If a tangent $p$ becomes tight at a time $t$ in the original setting, then
$t\ge t^*$.
\end{lemma}
\begin{proof}
Let $I_K(t)$ be the interval that contains $p$. Since this interval is
nonempty, $K\in \mathcal{K}$, and since the contribution to $p$ must be
increasing over time, $\sum_{i\in K} \delta_i > 0$.

Tangent $p$ is tight at time $t$ in the original setting, and therefore
$\omega_p(t) = f_p$. By Lemma \ref{lm:loc:oth-r:contrib-in},
$\omega_p^K(\tau_K) = f_p$, and hence $p$ is tight at time $\tau_K$ in
$A(K)$. It follows that $\tau_K \ge \tau'_K$, and therefore $t \ge t'_K \ge
t^*$.
\end{proof}
\begin{lemma}
\label{lm:loc:oth-r:tang-le}
Each tangent $p'_K$ with $\sum_{i\in K} \delta_i > 0$ becomes tight at a
time $t\le t'_K$ in the original setting.
\end{lemma}
\begin{proof}
Since $p'_K$ is tight at time $\tau'_K$ in $A(K)$, we have
$\omega_{p'_K}^K(\tau'_K) = f_{p'_K}$. By Lemma
\ref{lm:loc:oth-r:contrib-out}, $\omega_{p'_K}(t'_K) \ge f_{p'_K}$, which
means that $p'_K$ is tight or over-tight at time $t'_K$ in the original
setting. Therefore, $p'_K$ becomes tight at a time $t\le t'_K$ in the
original setting.
\end{proof}
\begin{lemma}
\label{lm:loc:oth-r:multi-sim}
Tangent $p^*$ becomes tight first in the original setting, at time $t^*$.
The quantities $p^*$ and $t^*$ can be computed in time $O(m^3)$.
\end{lemma}
\begin{proof}
By Lemma \ref{lm:loc:oth-r:tang-ge}, tangents only become tight at times
$t\ge t^*$, while by Lemma \ref{lm:loc:oth-r:tang-le}, $p^* =
p'_{\argmin\left\{ t'_K : \sum_{i\in K} \delta_i > 0, K\in \mathcal{K}
  \right\}}$ becomes tight at a time $t\le \min\bigl\{ t'_K : \sum_{i\in K}
\delta_i > 0, K\in \mathcal{K} \bigr\} = t^*$. Therefore, $p^*$ becomes
tight first, at time $t^*$.

Concerning the running time, note that the $\theta_r$ can be computed in
$O(m^2)$ and sorted in $O(m^2\log m)$ time. The permutations $\pi(+\infty)$
and $\pi(0)$ can be computed in $O(m\log m)$ time. Processing the $\theta_r$
in increasing order, we can compute each $\pi(\theta_r)$ in $O(m)$ time
amortized over all $\theta_r$. Computing $p'_K$ and $t'_K$ for all $K \in
\mathcal{K}(\theta_r)$ takes $O(m)$ time. Therefore, the total running time
is $O(m^3)$.
\end{proof}
The results in this section can be generalized further to allow the rates
$\delta_i$ to be negative.
\subsection{Analysis of Multiple Facilities}
\label{sect:loc:multi-fac}
We now show how to execute Algorithm \textsc{FLPD} implicitly when problem
\eqref{mp:loc:conc} has multiple facilities. In Section
\ref{sect:loc:single-fac}, in addition to assuming the presence of only one
facility, we assumed that the connection costs $c_{ij}$ were $0$. We remove
this assumptions as well.

In this section, we continue to refer to facilities of infinitely-sized
problem \eqref{ip:loc:pwl-inf} as tangents, and reserve the term facility
for facilities of concave cost problem \eqref{mp:loc:conc}. We say that
customer $i$ contributes to concave cost facility $j$ if $v_i \ge c_{ij}$.
We distinguish between when a customer contributes to a concave cost
facility $j$ and when the customer contributes to a tangent $p$ belonging to
concave cost facility $j$.

When executing Algorithm \textsc{FLPD} implicitly, the input consists of
$m$, $n$, the connection costs $c_{ij}$, the demands $d_i$, and the cost
functions $\phi_j$, given by an oracle. As intermediate variables, we
maintain the time $t$, and the vectors $v$, $x$, and $y$. For $x$ and $y$,
we maintain only the non-zero entries. The algorithm returns $v$, $x$, and
$y$. We also maintain standard data structures to manipulate these
quantities as necessary. Note that we do not maintain nor return the vector
$w$, as any one of its entries can be computed through the invariant
$w_{ijp} = \max\{0, v_i - (c_{ij} + s_{jp})d_i\}$.

Clearly, step \eqref{alg:flpd:start} can be executed in polynomial time. In
order to use induction, suppose that we have executed at most $m-1$
iterations of loop \eqref{alg:flpd:while} so far. Since the algorithm opens
at most one tangent at each iteration, at any point at most $m-1$ tangents
are open. To analyze step \eqref{alg:flpd:comp_t}, we consider three events
that may occur as this step is executed:
\begin{enumerate}
\item A closed tangent becomes tight.
\item An unconnected customer begins contributing to an open tangent.
\item An unconnected customer begins contributing to a facility. 
\end{enumerate}
When step \eqref{alg:flpd:comp_t} is executed, the time $t$ stops increasing
when event 1 or 2 occurs. For the purpose of analyzing this step, we assume
that $t$ increases to $+\infty$ and that $v_i$ for unconnected customers are
increased so as to maintain $v_i = td_i$.
\begin{lemma}
\label{lm:loc:min_comp}
Suppose that event $e$ at facility $j$ is the first to occur after the
beginning of step \eqref{alg:flpd:comp_t}. Then we can compute the time $t'$
when this event occurs in polynomial time.
\end{lemma}
\begin{proof}
If $e=1$, we use Lemma \ref{lm:loc:multi-sim} to compute $t'$. The lemma's
assumptions can be satisfied as follows. Since no events occur at other
facilities until time $t'$, we can assume that $j$ is the only facility.
Since the set of customers contributing to facility $j$ will not change
until time $t'$, we can satisfy the assumption that $c_{ij}=0$ by
subtracting $c_{ij}$ from each $v_i$ having $v_i \ge c_{ij}$. Since an
unconnected customer will not begin contributing to an open tangent until
time $t'$, we can assume that there are no open tangents. We can satisfy the
assumption that $t=0$ at the beginning of step \eqref{alg:flpd:comp_t} by
adding $td_i$ to each $v_i$.

If $e=2$, we compute $t'$ by iterating over all unconnected customers and
open tangents of facility $j$. If $e=3$, we compute $t'$ by iterating over
all unconnected customers.
\end{proof}
When other events occur between the beginning of step
\eqref{alg:flpd:comp_t} and time $t'$, the computation in this lemma may be
incorrect, however we can still perform it. Let $t'_e(j)$ be the time
computed in this manner for a given $e$ and $j$, and let $t^* = \min\{
t'_e(j) : e\in [3], j\in [n]\}$ and $(e^*, j^*) = \argmin\{ t'_e(j) : e\in
[3], j\in [n]\}$.
\begin{lemma}
\label{lm:loc:min_equiv}
Event $e^*$ at facility $j^*$ is the first to occur after the beginning of
step \eqref{alg:flpd:comp_t}. This event occurs at time $t^*$.
\end{lemma}
\begin{proof}
Suppose that an event $e'$ at a facility $j'$ occurs at a time $t' < t^*$.
If $e' \in \{2,3\}$, then $t' \ge t'_{e'}(j') \ge t^*$. This is a
contradiction, and therefore this case cannot occur.

If $e'=1$, then we consider two cases. If there is an event $e''\in \{2,3\}$
that occurs at a facility $j''$ at a time $t''<t'$, then we use $t'' \ge
t'_{e''}(j'') \ge t^*$ to obtain a contradiction. If there is no such event
$e''$, then no new customer begins contributing to facility $j'$ between the
beginning of step \eqref{alg:flpd:comp_t} and time $t'$. Therefore, $t' \ge
t'_{e'}(j')\ge t^*$, and we again obtain a contradiction.
\end{proof}
Once we have computed $t^*$, $e^*$, and $j^*$, we finish executing step
\eqref{alg:flpd:comp_t} as follows. If $e^*=3$, that is if the first event
to occur is an unconnected customer beginning to contribute to $j^*$, we
update the list of customers contributing to $j^*$ and recompute $t^*$,
$e^*$, and $j^*$. Since there are $n$ facilities and at most $m$ unconnected
customers, event 3 can occur at most $mn$ times before event 1 or 2 takes
place.

Once event 1 or 2 takes place, step \eqref{alg:flpd:comp_t} is complete, and
we have to execute step \eqref{alg:flpd:open} or \eqref{alg:flpd:conn}. It
is easy to see that these steps can be executed in polynomial time.
Therefore, an additional iteration of loop \eqref{alg:flpd:while} can be
executed in polynomial time. By induction, each of the first $m$ iterations
of loop \eqref{alg:flpd:while} can be executed in polynomial time. At each
iteration, an unconnected customer is connected, either in step
\eqref{alg:flpd:open} or \eqref{alg:flpd:conn}. Therefore, loop
\eqref{alg:flpd:while} iterates at most $m$ times. Obviously, step
\eqref{alg:flpd:ret} can be executed in polynomial time, and therefore
Algorithm \textsc{FLPD} can be executed implicitly in polynomial time.
Recall that we called the algorithm obtained by executing \textsc{FLPD}
implicitly on infinitely-sized problem \eqref{ip:loc:pwl-inf}
\textsc{ConcaveFLPD}.
\begin{theorem}
Algorithm \textsc{ConcaveFLPD} is a 1.61-approximation algorithm for
concave cost facility location, with a running time of $O(m^3n + mn\log n)$.
\end{theorem}
\begin{proof}
At the beginning of the algorithm, we sort the connection costs $c_{ij}$,
which can be done in $O(mn\log(mn))$ time. Next, we bound the time needed
for one iteration of loop \eqref{alg:flpd:while}. Note that since loop
\eqref{alg:flpd:while} iterates at most $m$ times, there are at most $m$
open tangents at any point in the algorithm.

In step \eqref{alg:flpd:comp_t}, we first compute $\min\{ t'_1(j) : j\in
[n]\}$. Computing each $t'_1(j)$ requires $O(m^2)$ per facility, and thus
this part takes $O(m^2 n)$ overall. Next, we compute $\min\{ t'_2(j) : j\in
[n]\}$, using $O(1)$ per customer and open tangent, and thus $O(m^2)$
overall. Finally, we compute $\min\{ t'_3(j) : j\in [n]\}$. Since $\min\{
t'_3(j) : j\in [n] \} = \min\{ c_{ij} : c_{ij} \ge t\}$, we have sorted the
values $c_{ij}$, and $t$ only increases as the algorithm runs, this
operation takes $O(mn)$ over the entire run of the algorithm. Therefore, we
can determine the next event to occur in $O(m^2n)$.

If event 1 or 2 is the next one, step \eqref{alg:flpd:comp_t} is complete.
If event 3 is next, an unconnected customer begins to contribute to facility
$j^*$. In this case, we recompute $t'_1(j^*)$ and $\min\{ t'_1(j) : j\in
[n]\}$. Recomputing $t'_1(j^*)$ can be done in $O(m)$, since we only have to
add one customer to the setting of Lemma \ref{lm:loc:multi-sim}. Recomputing
$\min\{ t'_1(j) : j\in [n]\}$ takes $O(1)$, as $t'_1(j^*)$ does not increase
when an unconnected customer begins contributing to facility $j^*$. Note
that $\min\{ t'_2(j) : j\in [n]\}$ does not change. Next, we recompute
$\min\{ t'_3(j) : j\in [n]\}$, which takes $O(mn)$ over the entire run of
the algorithm. The total time to process event 3 and determine the next
event to occur is $O(m)$. Event 3 occurs at most $mn$ times before event 1
or 2 occurs, and therefore the total time for processing event 3 occurrences
is $O(m^2n)$.

Step \eqref{alg:flpd:open} can be done in $O(m)$, and step
\eqref{alg:flpd:conn} in $O(1)$. Therefore, the time for one iteration of
loop \eqref{alg:flpd:while} is $O(m^2n)$. Since there are at most $m$
iterations of loop \eqref{alg:flpd:while}, the running time of the algorithm
is $O(m^3n + mn\log n)$.

By Theorem \ref{th:flpd:161}, Algorithm \textsc{FLPD} is a
1.61-approximation algorithm for problem \eqref{ip:loc:classic}. The
approximation ratio for problem \eqref{mp:loc:conc} follows directly from
the fact that we execute Algorithm \textsc{FLPD} implicitly on
infinitely-sized problem \eqref{ip:loc:pwl-inf}.
\end{proof}
By a similar application of our technique to the 1.861-approximation
algorithm for classical facility location of Mahdian et al.
\cite{MR2146253}, we obtain a 1.861-approximation algorithm for concave cost
facility location with a running time of $O(m^2n+mn\log n)$.
\section{Concave Cost Lot-Sizing}
\label{sect:lots}
In this section, we apply the technique developed in Section \ref{sect:loc}
to concave cost lot-sizing. The classical lot-sizing problem is defined in
Section \ref{sect:intro:lots}, and can be written as a linear program:
\begin{subequations}
\label{lp:lots:classic}
\begin{align}
\min\ &\sum_{s=1}^n f_s y_s +
\sum_{s=1}^n \sum_{t=s}^n (c_s+h_{st}) d_t x_{st},\\
\text{s.t.}\ &\sum_{s=1}^t x_{st} = 1, &1\le t\le n,\\
&0 \le x_{st} \le y_s, 
&1\le s\le t\le n.
\end{align}
\end{subequations}
Recall that $f_t\in \bbR_+$ and $c_t\in \bbR_+$ are the fixed and per-unit
costs of placing an order at time $t$, and $d_t\in \bbR_+$ is the demand at
time $t$. The per-unit holding cost at time $t$ is $h_t\in \bbR_+$, and for
convenience, we defined $h_{st} = \sum_{i=s}^{t-1} h_i$. Note that we omit
the constraints $y_s \in \{0,1\}$, as there is always an optimal extreme
point solution that satisfies them \cite{MR0525912}.

We now adapt the algorithm of Levi et al. \cite{MR2233997} to work in the
setting of problem (\ref{lp:lots:classic}). Levi et al. derive their
algorithm in a slightly different setting, where the costs $h_{st}$ are not
necessarily the sum of period holding costs $h_t$, but rather satisfy an
additional monotonicity condition.

The dual of problem (\ref{lp:lots:classic}) is given by:
\begin{subequations}
\label{lp:lots:classic-d}
\begin{align}
\max\ &\sum_{t=1}^n v_t, \label{lp:lots:classic-d:obj}\\
\text{s.t.}\ &v_t \le (c_s + h_{st}) d_t + w_{st}, 
&1\le s \le t \le n, \label{lp:lots:classic-d:v}\\
&\sum_{t=s}^n w_{st} \le f_{s}, &1\le s\le n,
\label{lp:lots:classic-d:w}\\
&w_{st}\ge 0, &1\le s\le t\le n.
\end{align}
\end{subequations}
As with facility location, since the variables $w_{st}$ do not appear in the
objective, we assume the invariant $w_{st} = \max\{0, v_t - (c_s + h_{st})
d_t\}$. Note that lot-sizing orders correspond to facilities in the facility
location problem, and lot-sizing demand points correspond to customers in
the facility location problem.

We refer to dual variable $v_t$ as the \emph{budget} of demand point $t$. If
$v_t \ge (c_s + h_{st}) d_t$, we say that demand point $t$
\emph{contributes} to order $s$, and $w_{st}$ is its contribution. The total
contribution received by an order $s$ is $\sum_{t=s}^n w_{st}$. An order $t$
is \emph{tight} if $\sum_{t=s}^n w_{st}=f_s$ and \emph{over-tight} if
$\sum_{t=s}^n w_{st} > f_s$.

The primal complementary slackness constraints are:
\begin{subequations}
\label{eq:lots:slack}
\begin{align}
&x_{st}(v_t - (c_s + h_{st}) d_t - w_{st}) = 0, &1\le s\le t\le n,
\label{eq:lots:slack:x}\\
&y_s\left(\sum_{t=s}^n w_{st} - f_s\right) = 0, &1\le s\le n.
\label{eq:lots:slack:y}
\end{align}
\end{subequations}
Let $(x,y)$ be an integral primal feasible solution, and $(v,w)$ be a dual
feasible solution. Constraint (\ref{eq:lots:slack:x}) says that demand point
$t$ can be served from order $s$ in the primal solution only if $s$ is the
closest to $t$ with respect to the modified costs $c_s + h_{st} + w_{st} /
d_t$. Constraint (\ref{eq:lots:slack:y}) says that order $t$ can be placed
in the primal solution only if it is tight in the dual solution.

The algorithm of Levi et al., as adapted here, starts with dual feasible
solution $(v,w)=0$ and iteratively updates it, while maintaining dual
feasibility and increasing the dual objective. At the same time, guided by
the primal complementary slackness constraints, the algorithm constructs an
integral primal solution. The algorithm concludes when the integral primal
solution becomes feasible. An additional postprocessing step decreases the
cost of the primal solution to the point where it equals that of the dual
solution. At this point, the algorithm has computed an optimal solution to
the lot-sizing problem.

We introduce the notion of a wave, which corresponds to the notion of time
in the primal-dual algorithm for facility location. In the algorithm, we
will denote the wave position by $W$, and it will decrease continuously from
$h_{1n}$ to $0$, and then possibly to a negative value not less than
$-c_1-f_1$. We associate to each step of the algorithm the wave position
when it occurred.
{\addtolength{\algrightmarginwidth}{0.5ex}
\addtolength{\algleftmarginwidth}{0.5ex}
\begin{algorithm}[H]
\small
\algname{Algorithm LSPD}{$n \in \mathbb{Z}_+$; $c,f,d\in \bbR^{n}_+, 
h\in \bbR^{n-1}_+$}
\begin{algtab}
Start with the wave at $W=h_{1n}$ and the dual solution $(v,w) = 0$. All orders
are closed, and all demand points are unserved, i.e. $(x,y)=0$.\\
\alglabel{alg:lspd:while}
\textbf{While} there are unserved demand points:\\
\algbegin \alglabel{alg:lspd:comp_t} 
Decrease $W$ continuously. At the same
time increase $v_t$ and $w_{st}$ for unserved demand points $t$ 
so as to maintain $v_t = \max\{ 0, d_t (h_{1t} - W) \}$ and 
$w_{st} = \max \{ 0, v_t - (c_s + h_{st}) d_t\}$. The wave stops when 
an order becomes tight.\\
\alglabel{alg:lspd:open} 
Open the order $s$ that became tight. For each
unserved demand point $t$ contributing to $s$, serve $t$ from $s$. \\ 
\algend
\alglabel{alg:lspd:for}
\textbf{For} each open order $s$ from $1$ to $n$:\\
\algbegin \alglabel{alg:lspd:perm_open} 
If there is a demand point $t$ that
contributes to $s$ and to another open order $s'$ with $s'<s$, close
$s$. Reassign all demand points previously served from $s$ to
$s'$.\\ \algend
Return $(x,y)$ and $(v,w)$.
\end{algtab}
\end{algorithm}}
\vskip -1em
In case of a tie between order points in step (\ref{alg:flpd:open}), we
break the tie arbitrarily. Depending on the demand points that remain
unserved, another one of the tied orders may open immediately in the next
iteration of loop (\ref{alg:flpd:while}).

The proof of the following theorem is almost identical to that from
\cite{MR2233997}, and therefore for this proof we assume the reader is
familiar with the lot-sizing results from \cite{MR2233997}.
\begin{theorem}
\label{th:lspd}
Algorithm \textsc{LSPD} is an exact algorithm for the classical lot-sizing
problem.
\end{theorem}
\begin{proof}
We will show that after we have considered open order $s$, at the end of
step (\ref{alg:lspd:perm_open}), we maintain two invariants. First, each
demand point is contributing to the fixed cost of at most one open order
from the set $\{1, \dots, s\}$. Second, each demand point is assigned to an
open order and contributes to its fixed cost.

The first invariant follows from the definition of the algorithm. Indeed,
if a demand point $t'$ is contributing to $s'$ and $s$ with $s'<s$, then the
algorithm would have closed $s$.

Clearly the second invariant holds at the beginning of loop
(\ref{alg:lspd:for}). It continues to hold after we review order $s$ if we
have not closed $s$. Let us now consider the case when we have closed
$s$. The demand points that have contributed to $s$ can be classified into
two categories. The first category contains the demand points whose dual
variables stopped due to $s$ becoming tight---these demand points were
served from $s$ and are now served from $s'$. Since $t$ contributes to $s'$,
so do these demand points. The second category contains the demand points
whose dual variables stopped due to another order $s''$ becoming tight. The
case $s''<s$ cannot happen, or $s$ would have never opened.  Hence, $s<s''$,
and therefore $s''$ is currently open. Moreover, these demand points are
currently served from $s''$ and are contributing to it.

Therefore, at the end of loop (\ref{alg:lspd:for}), each demand point is
contributing to the fixed cost of at most one open order. Therefore, the
fixed cost of opening orders is fully paid for by the dual
solution. Moreover, each demand point is served from an open order, and
therefore the primal solution is feasible. Since each demand point
contributes to the fixed cost of the order it is served from, the holding
and variable connection cost is also fully paid for by the dual
solution. Since the primal and dual solutions have the same cost, the
algorithm is exact.
\end{proof}
\subsection{Applying the Technique}
\label{sect:lots:techn}
We now proceed to develop an exact primal-dual algorithm for concave cost
lot-sizing. The concave cost lot-sizing problem is defined in Section
\ref{sect:intro:lots}:
\begin{subequations}
\label{mp:lots:conc}
\begin{align}
\min\ &\sum_{s=1}^n \phi_s\left(\sum_{t=s}^n d_t x_{st}\right) +
\sum_{s=1}^n \sum_{t=s}^n h_{st} d_t x_{st},\\
\text{s.t.}\ &\sum_{s=1}^t x_{st} = 1, &1 \le t\le n,\\
&x_{st} \ge 0, &1 \le s \le t \le n.
\end{align}
\end{subequations}
Here, the cost of placing an order at time $t$ is given by a nondecreasing
concave cost function $\phi_t : \bbR_+ \to \bbR_+$. We assume without loss
of generality that $\phi_t(0) = 0$ for all $t$.

The application of our technique to the lot-sizing problem is similar to its
application to the facility location problem in Section \ref{sect:loc}.
First, we reduce concave cost lot-sizing problem (\ref{mp:lots:conc}) to the
following infinitely-sized classical lot-sizing problem.
\begin{subequations}
\label{lp:lots:pwl-inf}
\begin{align}
\min\ &\sum_{s=1}^n \sum_{p\ge 0} f_{sp} y_{sp} +
\sum_{s=1}^n \sum_{t=s}^n \sum_{p\ge 0} (c_{sp}+h_{st}) d_t x_{spt},\\
\text{s.t.}\ &\sum_{s=1}^t \sum_{p\ge 0} x_{spt} = 1, \hskip 10em 1\le t\le n,\\
&0 \le x_{spt} \le y_{sp}, \hskip 8.5em 1\le s\le t\le n, p\ge 0.
\end{align}
\end{subequations}
Again we note that since LP (\ref{lp:lots:pwl-inf}) is infinitely-sized,
strongly duality does not hold automatically for it and its dual. However,
the proof of Algorithm \textsc{LSPD} relies only on weak duality. The fact
that the algorithm produces a primal solution and a dual solution with the
same cost implies that both solutions are optimal and that strong duality
holds.

Following Section \ref{sect:loc}, let \textsc{ConcaveLSPD} be the algorithm
obtained by executing Algorithm \textsc{LSPD} implicitly on infinitely-sized
problem \eqref{lp:lots:pwl-inf}.
\begin{theorem}
Algorithm \textsc{ConcaveLSPD} is an exact algorithm for concave cost
lot-sizing, with a running time of $O(n^2)$.
\end{theorem}
\begin{proof}
We consider the following events that may occur as step
\eqref{alg:lspd:comp_t} of Algorithm \textsc{LSPD} is executed:
\begin{center}
\small
\begin{tabular}{cll}
Time & & Event \\
\cline{1-1} \cline{3-3}
\rule{0pt}{2.5ex}$W_1(t)$ & & The wave reaches demand point $t$, 
i.e. $W=h_{1t}$. \\
$W_2(t)$ & & A tangent $p$ of order point $t$ becomes tight. \\
\end{tabular}
\end{center}
If for an order point $t$, no tangents become tight in the course of the
algorithm, we let $W_2(t) = +\infty$. The wave positions $W_1(t)$ can be
computed for all $t$ at the beginning of the algorithm in $O(n)$. 

We compute the positions $W_2(t)$ by employing a set of intermediate values
$W'_2(t)$. Each value $W'_2(t)$ is defined as the time when a tangent of
order point $t$ becomes tight in a truncated problem consisting of time
periods $t, t+1, \dots, n$. We compute a subset of these values as follows.
First, we compute $W'_2(n)$, which requires $O(1)$ time by Lemma
\ref{lm:loc:uni-sim}. To compute $W'_2(t)$ given that $W'_2(t+1), \dots,
W'_2(n)$ are computed, we can employ Lemma \ref{lm:loc:multi-sim}.

The dual variables representing demand points $t, \dots, n$ can be divided
into three consecutive intervals. First are the dual variables that are
increasing at the same rate as part of the wave, then the dual variables
$v_k$ that are not increasing but exceed $h_{tk}$, and finally the dual
variables $v_k$ that are not increasing, do not exceed $h_{tk}$, and
therefore play no role in this computation. We employ Lemma
\ref{lm:loc:multi-sim} and distinguish two cases:
\begin{enumerate}
\item
Lemma \ref{lm:loc:multi-sim} can be used to detect if a tangent is
overtight. This indicates that $W'_2(t)$ is an earlier wave position than
$W'_2(t+1), \dots, W'_2(k)$ for some $k$. In this case, we delete
$W'_2(t+1)$ from our subset and repeat the computation of $W'_2(t)$ as if
order point $t+1$ does not exist.
\item
There are no overtight tangents. Thus, a tangent becomes tight at a wave
position less than or equal to $W'_2(t+1)$. In this case we set $W'_2(t)$ to
this wave position, and proceed to the computation of $W'_2(t-1)$.
\end{enumerate}
After computing $W'_2(t)$, consider the values that remain in our subset and
denote them by $W'_2(t), W'_2(\pi(1)), \dots, W'_2(\pi(k))$ for some $k$. By
induction, these values yield the correct times when tangents become tight
for the truncated problem consisting of time periods $t, \dots, n$. After we
have computed $W'_2(1)$, the values $W'_2(t)$ remaining in our subset yield
the correct times $W_2(t)$, with the other values $W_2(t)=+\infty$.
Therefore, loop (\ref{alg:lspd:while}) is complete.

A computation by Lemma \ref{lm:loc:multi-sim} requires $O(n^2)$ time in the
worst case. Since in this setting, all dual variables that are increasing
exceed all dual variables that are stopped, each $W'_2(t)$ can be computed
by Lemma \ref{lm:loc:multi-sim} in $O(n)$. Each time we use Lemma
\ref{lm:loc:multi-sim} for a computation, a value $W'_2(t)$ is either
removed from the list or inserted into the list. Since each value is
inserted into the list only once, the total number of computations is
$O(n)$, and the total running time for loop (\ref{alg:lspd:while}) is
$O(n^2)$.

At the beginning of step (\ref{alg:lspd:for}), there are at most $n$ open
tangents, and $n$ demand points, and therefore this loop can be implemented
in $O(n^2)$ as well.
\end{proof}
Note that the values $W_2(t)$ also yield a dual optimal solution to the
infinitely-sized LP. The solution can be computed from the $W_2(t)$-s in
time $O(n)$ by taking $v_t = h_{1t} -  W_2(\sigma(t))$, where
$\sigma(t)$ is the latest time period less than or equal to $t$ that has
$W_2(\sigma(t)) < +\infty$.
\section{Concave Cost Joint Replenishment}
\label{sect:jrp}
In this section, we apply our technique to the concave cost joint
replenishment problem (JRP). The classical JRP is defined in Section
\ref{sect:intro:jrp}, and can be formulated as an integer program:
\begin{subequations}
\label{ip:jrp:classic}
\begin{align}
\min\ &\sum_{s=1}^n f^0 y^0_{s} + \sum_{s=1}^n \sum_{k=1}^K \ f^k y^k_s +
\sum_{s=1}^n \sum_{t=s}^n \sum_{k=1}^K h^k_{st} d^k_t x^k_{st},\\
\text{s.t.}\ &\sum_{s=1}^t x^k_{st} = 1, \hskip 10em 
1 \le t\le n, k\in [K],\\
&0 \le x^k_{st} \le y^0_{s}, \hskip 8.75em 1 \le s \le t \le n, 
k\in [K], \label{ip:jrp:classic:xy0}\\
&0 \le x^k_{st} \le y^k_{s}, \hskip 8.75em 1 \le s \le t \le n, 
k\in [K], \label{ip:jrp:classic:xyk}\\
&y^0_{s}\in \{0,1\}, y^k_{s}\in \{0,1\}, \hskip 5.25em 1\le s\le n, 
k\in [K]. \label{ip:jrp:classic:y01}
\end{align}
\end{subequations}
Recall that $f^0\in \bbR_+$ is the fixed joint ordering cost, $f^k\in
\bbR_+$ is the fixed individual ordering cost for item $k$, and $d^k_t\in
\bbR_+$ is the demand for item $k$ at time $t$. The per-unit holding cost
for item $k$ at time $t$ is $h^k_t$; for convenience we defined $h^k_{st} =
\sum_{i=s}^{t-1} h^k_i$.

The concave cost JRP, also defined in Section \ref{sect:intro:jrp}, can be
written as a mathematical program as follows:
\begin{subequations}
\label{mp:jrp:conc}
\begin{align}
\min\ &\sum_{s=1}^n \phi^0\left(\sum_{t=s}^n 
\sum_{k=1}^K d^k_t x^k_{st}\right) + \sum_{s=1}^n \sum_{k=1}^K \phi^k 
\left( \sum_{t=s}^n d^k_t x^k_{st} \right) + 
\sum_{s=1}^n \sum_{t=s}^n \sum_{k=1}^K h^k_{st} d^k_t x^k_{st},\\
\text{s.t.}\ &\sum_{s=1}^t x^k_{st} = 1, \hskip 12em 1 \le t\le n, 
k\in [K],\\
&x^k_{st} \ge 0, \hskip 13em 1 \le s \le t \le n, k\in [K].
\end{align}
\end{subequations}
Here the individual ordering cost for item $k$ at time $t$ is given by a
nondecreasing concave function $\phi^k : \bbR_+ \to \bbR_+$. We assume
without loss of generality that $\phi^k(0) = 0$ for all $k$. The joint
ordering cost at time $t$ is given by the function $\phi^0 : \bbR_+ \to
\bbR_+$. To reflect the fact that only the individual ordering costs are
general concave, $\phi^0$ has the form $\phi^0(0) = 0$ and $\phi^0(\xi) =
f^0$ for $\xi>0$.

Consider the case when the individual ordering cost functions $\phi^k$ are
piecewise linear with $P$ pieces:
\begin{equation}
\label{eq:jrp:pwl-def}
\phi^k(\xi^k_t) = 
\begin{cases}
\min \{  f^k_p + c^k_p \xi^k_t : p\in [P] \}, &\xi^k_t > 0,\\
0, &\xi^k_t = 0,
\end{cases}
\end{equation}
Unlike with concave cost facility location and concave cost lot-sizing, the
piecewise-linear concave cost JRP does not reduce polynomially to the
classical JRP. Since there are multiple items, different pieces of the
individual ordering cost functions $\phi^k$ may be employed by different
items $k$ as part of the same order at time $t$. When each cost function
consists of $P$ pieces, we would need $P^K$ time periods to represent each
possible combination, thereby leading to an exponentially-sized IP
formulation.

We could devise a polynomially-sized IP formulation for the piecewise-linear
concave cost JRP, however such a formulation would have a different
structure from the classical JRP, and would not enable us to apply our
technique together with the primal-dual algorithm of Levi et al.
\cite{MR2233997} for the classical JRP. Instead, we reduce the
piecewise-linear concave cost JRP to the following exponentially-sized
integer programming formulation, which we call the \emph{generalized joint
  replenishment} problem. Let $\pi=(p_1, \dots, p_K)$, and let $[P]^K = \{
(p_1, \dots, p_K) : p_i\in [P]\}$.
\begin{subequations}
\label{ip:jrp:gener}
\begin{align}
\min\ &\sum_{\substack{ s\in [n]\\ \pi \in [P]^{K}}} f^0 y^0_{s\pi} +
\sum_{\substack{s\in [n], k\in [K]\\ \pi \in [P]^{K}}}\ 
f^k_{p_k} y^k_{s\pi} +
\sum_{\substack{1\le s\le t\le n\\ k\in [K], \pi\in [P]^{K}}}
(c^k_{p_k} + h^k_{st}) d^k_t x^k_{s\pi t},\\
\text{s.t.}\ &\sum_{\substack{s\in [t]\\ \pi \in [P]^{K}}}
x^k_{s\pi t} = 1, \hskip 12.5em 1 \le t\le n, k\in [K],\\
&0 \le x^k_{s\pi t} \le y^0_{s\pi}, \hskip 10em 1 \le s \le t \le n,
k\in [K], \pi\in [P]^{K}, \label{ip:jrp:gener:xy0}\\
&0 \le x^k_{s\pi t} \le y^k_{s\pi}, \hskip 10em 1 \le s \le t \le n, 
k\in [K], \pi\in [P]^{K}, \label{ip:jrp:gener:xyk}\\
&y^0_{s\pi}\in \{0,1\}, y^k_{s\pi}\in \{0,1\}, \hskip 6.75em 1\le s\le n, 
k\in [K], \pi\in [P]^{K}. \label{ip:jrp:gener:y01}
\end{align}
\end{subequations}
The intuition behind the generalized JRP is that each time period $t$ in the
piecewise-linear concave cost JRP corresponds to $P^K$ time periods
$(t,\pi)$ in the generalized JRP. Each time period $(t,\pi)$ allows us to
use a different combination $\pi=(p_1, \dots, p_K)$ of pieces of the
individual order cost functions $\phi^1, \dots, \phi^K$.

This formulation does not satisfy the cost assumptions required for the
2-approximation algorithm of Levi et al. \cite{MR2233997}. In the next
section, we will devise, starting from the algorithm of Levi et al, an
algorithm for the generalized JRP that provides a 4-approximation guarantee
and runs in exponential time. In Section \ref{sect:jrp:techn}, we will
employ our technique to obtain a strongly polynomial 4-approximation
algorithm for the concave cost JRP.
\subsection{An Algorithm for the Generalized JRP}
\label{sect:jrp:gener}
Consider the LP relaxation of IP (\ref{ip:jrp:gener}) obtained by replacing
the constraints $y^0_{s\pi}\in \{0,1\},\linebreak[0] y^k_{s\pi}\in \{0,1\}$
with $y^0_{s\pi}\ge 0, y^k_{s\pi}\ge 0$. The dual of this LP relaxation is:
\begin{subequations}
\label{lp:jrp:gener-d}
\begin{align}
\max\ &\sum_{k=1}^K \sum_{t=1}^n v^k_t, \label{lp:jrp:gener-d:obj}\\
\text{s.t.}\ &v^k_t \le (c^k_{p_k} + h^k_{st}) d^k_t + 
w^k_{s\pi t} + u^k_{s\pi t}, 
&\substack{1\le s \le t \le n, k\in [K],\\ \pi\in [P]^{K},}
\label{lp:jrp:gener-d:v}\\
&\sum_{t=s}^n w^k_{s\pi t} \le f^k_{p_k}, 
&\substack{1\le s\le n, k\in [K],\\ \pi\in [P]^{K},}
\label{lp:jrp:gener-d:w}\\
&\sum_{k=1}^K \sum_{t=s}^n u^k_{s\pi t} \le f^0, 
&1\le s\le n, \pi\in [P]^{K}, \label{lp:jrp:gener-d:u}\\
&w^k_{s\pi t}\ge 0, u^k_{s\pi t}\ge 0, 
&\substack{1\le s\le t\le n, k\in [K],\\ \pi\in [P]^{K}.}
\label{lp:jrp:gener-d:ge0}
\end{align}
\end{subequations}
Since now both $w^k_{s\pi t}$ and $u^k_{s\pi t}$ are not present in the
objective, the invariants for them become more involved. When $\sum_{t=s}^n
\max\{0, v^k_t - (c^k_{p_k} + h^k_{st}) d^k_t \} \le f^k_{p_k}$, we let as
before
\begin{subequations}
\label{eq:p-d:jrp-invar}
\begin{equation}
w^k_{s\pi t} = \max\{0, v^k_t - (c^k_{p_k} + h^k_{st}) d^k_t \} \le
f^k_{p_k}.
\end{equation}
When $\sum_{t=s}^n \max\{0, v^k_t - (c^k_{p_k} + h^k_{st}) d^k_t \} >
f^k_{p_k}$, the algorithm will have fixed the values $w^k_{s\pi t}$ at the
point when $\sum_{t=s}^n \max\{0, v^k_t - (c^k_{p_k} + h^k_{st}) d^k_t \} =
f^k_{p_k}$. In this situation, we let
\begin{equation}
u^k_{s\pi t} = \max\{ 0, v^k_t - (c^k_{p_k} + h^k_{st}) d^k_t -
w^k_{s\pi t} \}.
\end{equation}
\end{subequations}

We now have demand points for every time-item pair, and we refer to $v^k_t$
as the \emph{budget} of item $k$ at time $t$. Given $\pi$, if $v^k_t \ge
(c^k_{p_k} + h^k_{st}) d_t$, we say that demand point $(k,t)$
\emph{contributes} to the fixed cost of individual order $(s,k,\pi)$ and
$w^k_{s\pi t}$ is its contribution. If $v^k_t \ge (c^k_{p_k} + h^k_{st})
d_t$ and $\sum_{t=s}^n w^k_{s\pi t} = f^k_{p_k}$, we say that demand point
$(k,t)$ contributes to the fixed cost of joint order $(s,\pi)$ and
$u^k_{s\pi t}$ is its contribution.

Since we now have several items, each with its own holding costs, we think
of $W$ as a ``master'' wave, and decrease it from $n$ to $1$ and then to a
bounded amount below $1$. For each item $k$, we maintain an item wave
\begin{equation}
\label{eq:jrp:wave-k}
W^k = h_{1\lfloor W\rfloor} + h_{\lfloor W\rfloor} (W - \lfloor
W \rfloor).
\end{equation}
Intuitively, the $W^k$ are computed so that the item waves arrive together
at time periods $1, \dots, n-1$ and advance linearly inbetween.
{
%\addtolength{\algrightmarginwidth}{-1ex}
%\addtolength{\algleftmarginwidth}{-1ex}
\begin{algorithm}[H]
\small
\algname{Algorithm JRPPD}{$n,K,P\in \mathbb{Z}_+$; 
$f^0\in \bbR_+$; $f,c\in \bbR^{KP}_+$; \\
\phantom{12345678901234567890} $d\in \bbR^{nK}_+$;
$h\in \bbR^{(n-1)K}_+$}
\begin{algtab}
Start with the wave at $W=n$ and the dual solution $(v,w,u) = 0$. 
All orders are closed, and all demand points are unserved, i.e. 
$(x,y)=0$.\\
\alglabel{alg:jrppd:while}
\textbf{While} there are unserved demand points:\\
\algbegin \alglabel{alg:jrppd:comp_t} 
Decrease $W$ continuously and update
$W^k$ according to (\ref{eq:jrp:wave-k}). At the same time, for unserved
demand points $(t,k)$, increase $v^k_t = \max \{ 0, d_t^k (h_{1t} - W^k)$,
and update $w_{s\pi t}$ and $u_{s\pi t}$ so as to maintain
(\ref{eq:p-d:jrp-invar}). The wave stops when a joint or individual order
becomes tight.\\
\alglabel{alg:jrppd:fix-inv} 
If an individual order $(s,k,\pi)$ became
tight, fix the variables $w^k_{s\pi t}$ as described in
(\ref{eq:p-d:jrp-invar}). If the joint order $(s,\pi)$ is also tight, serve
all demand points contributing to $(s,k,\pi)$ from $(s,\pi)$.\\
\alglabel{alg:jrppd:open} 
If a joint order $(s,\pi)$ became tight, open the
joint order and all tight individual orders $(s,\pi,k)$. For each unserved
demand point $(t,k)$ that contributes to joint order $(s,\pi)$, 
serve $(t,k)$ from $(s,\pi)$. \\ \algend
\alglabel{alg:jrppd:for}
\textbf{For} each open joint order $s$ from $1$ to $n$:\\
\algbegin \alglabel{alg:jrppd:perm_open} 
If there is a demand point $(t,k)$
that contributes to $s$ and to another open joint order $s'$ with $s'<s$,
close $s$. \\ \algend
\alglabel{alg:jrppd:move-items-for}
\textbf{For} each item $k$: \\
\algbegin \alglabel{alg:jrppd:move-items-while}
\textbf{While} not all demand points have been processed
in step (\ref{alg:jrppd:move-items-two}): \\
\algbegin \alglabel{alg:jrppd:move-items-one} 
Select the latest such demand
point $(t,k)$. Let $\mathrm{freeze}(t,k)$ be the location of $W^k$ when
$v_t^k$ was stopped, and let $s$ be the earliest open joint order in
$[\mathrm{freeze}(t,k),t]$. \\
\alglabel{alg:jrppd:move-items-two}
Open individual order $(s,k)$. Serve all demand points $(t',k)$ with $s
\le t'\le t$ from $(s,k)$. \\
\algend
\algend
Return $(x,y)$ and $(v,w)$.
\end{algtab}
\end{algorithm}}
\vskip -1em
A direct implementation of this algorithm will have an exponential running
time. It is possible to implement this algorithm to have a polynomial
running time, however we will not do so here. Instead, we only prove that it
provides a 4-approximation guarantee. The proof closely resembles that from
\cite{MR2233997}, and therefore for this proof we assume the reader is fully
familiar with the joint replenishment results from \cite{MR2233997}.
\begin{theorem}
Algorithm \textsc{JRPPD} provides a 4-approximation guarantee for the
generalized JRP.
\end{theorem}
\begin{proof}
First, similarly to the proof of Levi et al. and Theorem \ref{th:lspd},
after loop (\ref{alg:jrppd:for}), each demand point contributes to at most
one open joint order. Since we do not open any other joint orders after this
step, the joint order cost is fully paid by the dual solution, i.e.
$\sum_{s=1}^n \sum_{\pi \in [P]^{K}} f^0 y^0_{s\pi} \le \sum_{k=1}^K
\sum_{t=1}^n v^k_t$. Out of 4 times the cost of the dual solution, we
allocate one toward the cost of the joint orders. Therefore, we need not
consider the cost of the joint orders further in this proof.

Second, also similarly to the proof of Levi et al. and Theorem \ref{th:lspd},
after loop (\ref{alg:jrppd:for}), for each demand point $(t,k)$ there is at
least one open joint order in $[\mathrm{freeze}(t,k), t]$. Therefore, after
loop (\ref{alg:jrppd:move-items-for}), the algorithm produces a feasible
primal solution.

Since we have already covered the cost of joint orders, we now consider each
item $k$ separately. We bound the holding cost and the cost of individual
orders in terms of the dual value, similarly to Levi et al. Due to the
different cost structure of the JRP and generalized JRP, we are only able to
bound the holding and individual order cost by 3 times the cost of the dual
solution, i.e. $\sum_{s=1}^n \sum_{\pi \in [P]^{K}} f^k_{p_k} y^k_{s\pi} +
\sum_{s=1}^n \sum_{t=s}^n \sum_{\pi\in [P]^{K}} (c^k_{p_k} + h^k_{st}) d^k_t
x^k_{s\pi t} \le 3 \sum_{t=1}^n v^k_t$.

Therefore, we obtain a 4 approximation algorithm.
\end{proof}
\subsection{Applying the Technique}
\label{sect:jrp:techn}
Finally, we obtain the strongly polynomial algorithm for the concave cost
JRP. First, we reduce the concave cost JRP to an infinitely-sized
generalized JRP:
\begin{subequations}
\label{ip:jrpd:gener-inf}
\begin{align}
\min\ &\sum_{\substack{ s\in [n]\\ \pi \in \bbR_+^K}} 
f^0 y^0_{s\pi} + 
\sum_{\substack{s\in [n], k\in [K]\\ \pi \in \bbR_+^K}}\ 
f^k_{p_k} y^k_{s\pi} + 
\sum_{\substack{1\le s\le t\le n\\ k\in [K], \pi\in \bbR_+^K}}
(c^k_{p_k} + h^k_{st}) d^k_t x^k_{s\pi t},\\
\text{s.t.}\ &\sum_{\substack{s\in [t]\\ \pi \in \bbR_+^K}}
x^k_{s\pi t} = 1, \hskip 12.5em 1 \le t\le n, k\in [K],\\
&0 \le x^k_{s\pi t} \le y^0_{s\pi}, \hskip 10em 1 \le s \le t \le n, 
k\in [K], \pi\in \bbR_+^K, \label{ip:jrpd:gener-inf:xy0}\\
&0 \le x^k_{s\pi t} \le y^k_{s\pi}, \hskip 10em 1 \le s \le t \le n, 
k\in [K], \pi\in \bbR_+^K, \label{ip:jrpd:gener-inf:xyk}\\
&y^0_{s\pi}\in \{0,1\}, y^k_{s\pi}\in \{0,1\}, \hskip 6.75em 
1\le s\le n, k\in [K], \pi\in \bbR_+^K. 
\label{ip:jrpd:gener-inf:y01}
\end{align}
\end{subequations}
As before, let \textsc{ConcaveJRPPD} be the algorithm obtained by executing
Algorithm \textsc{JRPPD} implicitly on infinitely-sized problem
\eqref{ip:jrpd:gener-inf}.
\begin{theorem}
Algorithm \textsc{ConcaveJRPPD} is a strongly polynomial 4-approximation
algorithm for the concave cost JRP.
\end{theorem}
\begin{proof}
Although in this setting all ordering costs are the same over time, we
will need to refer to ordering costs and groups of tangents at specific times.
With this in mind, we will refer to the ordering cost of item $k$ at time $t$
by $(\phi^k,t)$ and to the joint ordering cost at time $t$ by $(\phi^0,t)$.

Note that we do not need to keep track of variables for each $\pi \in
\bbR_+^K$ explicitly. Denote the tangent to the individual ordering cost
$(\phi^k,s)$ that becomes tight first by $p^*_{ks}$. Then all the other
tangents to this individual ordering cost at this time are no longer
relevant:
\begin{enumerate}
\item
Concerning individual ordering costs. For any item $l\ne k$, the behavior of
demand points $(t,l)$ or tangents to costs $(\phi^l,t)$ does not depend on
item $k$, except through the joint ordering cost.
\item
Concerning the joint ordering cost. For any wave position, the contribution
to the joint ordering cost $\sum_{k=1}^K \sum_{t=s}^n u^k_{s\pi t}$ is
highest for $\pi$ with $p_k = p^*_{ks}$.
\end{enumerate}
Therefore, it suffices to keep track, for each item $k$ and time $s$, of the
wave position when the first tangent to $(\phi^k,s)$ becomes tight. When
this occurs, we can stop considering all other tangents to $(\phi^k,s)$.
When computing the wave position when the joint ordering cost becomes tight,
we need to consider only the tangents that became tight for individual
ordering costs $(\phi^k,s)$. Through this transformation, the wave position
when the joint ordering cost becomes tight can be computed by Lemma
\ref{lm:loc:multi-sim}.

We now define the following events and wave positions when they
occurred:
\begin{center}
\small
\begin{tabular}{lll}
Wave Pos. & & Event \\
\cline{1-1} \cline{3-3}
\rule{0pt}{2.5ex}$W_1(t)$ & & The wave reaches time period $t$, 
i.e. $W=t$. \\
$W_2(t,k)$ & & A tangent $p$ of order point $(t,k)$ becomes tight. \\
$W_3(t)$ & & The joint order at time $t$ becomes tight.
\end{tabular}
\end{center}
The computation now proceeds similarly to the lot-sizing case. We compute
the largest of the wave positions $W_1(t)$, $W_2(t,k)$, and $W_3(t)$ (which
corresponds to the smallest time in the facility location problem). After
the computation we update the other $W$-values, and iterate.
\end{proof}
\section*{Acknowledgments}
This research was supported in part by the Air Force Office of Scientific
Research, Amazon.com, and the Singapore-MIT Alliance.
%
%
% __________________________________________________________________________
% Back matter
%
\bibliographystyle{alpha}
\bibliography{strongly}
\end{document}